\documentclass[a4paper,11pt]{article}

\usepackage{amssymb, amsmath,amsthm}
\newtheorem{theorem}{Theorem}[section]
\newtheorem{corollary}[theorem]{Corollary}
\newtheorem{lemma}[theorem]{Lemma}
\newtheorem{proposition}[theorem]{Proposition}
\newtheorem{definition}[theorem]{Definition}
\newtheorem{remark}[theorem]{Remark}
\numberwithin{equation}{section}
\numberwithin{equation}{section}

\newcommand\at{\allowbreak}

\title{Kawasaki dynamics in continuum: micro- and mesoscopic descriptions}

\author{Christoph Berns\thanks{Fakult\"at f\"ur Mathematik, Universit\"at Bielefeld, Postfach 100 131, 33501 Bielefeld, Germany ({\tt
 cberns@\at math.\at uni-bielefeld.\at de}).} \and Yuri Kondratiev\thanks{Fakult\"{a}t
 f\"{u}r Mathematik, Universit\"{a}t Bielefeld, Postfach 100 131, 33501 Bielefeld,
 Germany ({\tt kondrat@\at math.\at uni-bielefeld.\at de})}\and Yuri Kozitsky\thanks{Instytut Matematyki, Uniwersytet Marii Curie-Sk{\l}odowsiej, 20-031 Lublin, Poland ({\tt jkozi@\at hektor.\at umcs.\at lublin.\at de})}\and Oleksandr
 Kutoviy\thanks{Fakult\"{a}t f\"{u}r Mathematik, Universit\"{a}t
 Bielefeld, Postfach 100 131, 33501 Bielefeld, Germany ({\tt
 kutoviy@\at math.\at uni-bielefeld.\at de}).}}

\begin{document}

\maketitle

\begin{abstract}

The dynamics of an infinite system of point particles in $\mathbb{R}^d$, which hop
and interact with each other, is
described at both micro- and mesoscopic levels.
The states of the system are probability measures on
the space of  configurations of particles.
For a bounded time interval $[0,T)$, the evolution of states $\mu_0 \mapsto \mu_t$ is shown to
hold in a space of sub-Poissonian measures. This result is obtained by:
(a) solving equations for correlation functions, which yields the evolution $k_0 \mapsto k_t$, $t\in [0,T)$, in a scale
of Banach spaces; (b) proving that each $k_t$ is a correlation function for a unique measure $\mu_t$.
The mesoscopic theory is based on
a Vlasov-type  scaling, that yields a mean-field-like approximate description in terms of the particles' density which
obeys
a kinetic equation.
The latter equation is rigorously derived from
that for the correlation functions  by the scaling procedure. We prove that the kinetic equation has a unique solution
$\varrho_t$, $t\in [0,+\infty)$.
\end{abstract}

\section{Introduction}
\subsection{The setup}

In this paper, we study the dynamics of an infinite system of point particles in $\mathbb{R}^d$ which hop
and interact with each other. The corresponding phase space is the set of configurations
\begin{equation}
\label{C100}
 \Gamma=\{\gamma\subset\mathbb R^d :
 |\gamma\cap K|<\infty\text{ for any compact $K\subset\mathbb R^d$
 }\},
\end{equation}
where $|A|$ denotes the cardinality of a finite set $A$.
The set $\Gamma$ is equipped with a complete metric and with the corresponding Borel $\sigma$-field, which allows
one to employ probability measures on $\Gamma$.

In this  work, we follow
the statistical approach to stochastic dynamics, see e.g.,  \cite{Dima0,DimaN,Dima2} and the literature quoted in those articles.  In this approach, a model is specified by a Markov `generator', which acts on observables -- appropriate functions $F:\Gamma \to \mathbb{R}$. For the model considered here, it has the form
\begin{equation}
\label{C2} (LF) (\gamma) = \sum_{x\in \gamma}\int_{\mathbb{R}^d}
c(x,y,\gamma) \left[ F(\gamma \setminus x \cup y) - F(\gamma)\right]
d y, \quad \gamma \in \Gamma.
\end{equation}
In (\ref{C2}) and in the sequel in the
corresponding context, we treat each $x\in \mathbb{R}^d$ also as a
single-point configuration $\{x\}$. That is, if $x$ belongs to $\gamma$ (resp. $y$ does not), by
$\gamma\setminus x$ (resp. $\gamma \cup y$) we mean the configuration which is obtained from $\gamma$ by removing
$x$ (resp. by adding $y$).  The elementary act of the dynamics described by (\ref{C2}), which with probability
$ c(x,y,\gamma)dt$ occurs during the infinitesimal time $dt$, consists in a random change from $\gamma$ to $\gamma\setminus x \cup y$. The rate $c(x,y,\gamma)$ may depend  on $z\in \gamma$ with $z\neq x,y$, which is interpreted as an interaction
 of particles. In this article, we choose
\begin{equation}
  \label{C3}
c(x,y,\gamma) = a (x-y) \exp\left(- E^\phi (y, \gamma) \right),
\end{equation}
where the jump kernel $a: \mathbb{R}^d \rightarrow [0,+\infty) =:  \mathbb{R}_+$ is
such that $a(x) = a(-x)$ and
\begin{equation}
  \label{C4}
 \alpha : = \int_{\mathbb{R}^d} a(x) d x < \infty.
\end{equation}
The second factor in (\ref{C3}) describes the interaction, which is supposed to be pair-wise and repulsive. This means that
\begin{equation}
 \label{C5}
E^\phi (y, \gamma) = \sum_{z\in \gamma} \phi(y-z) \geq 0,
\end{equation}
where the `potential'
$\phi: \mathbb{R}^d \rightarrow  \mathbb{R}_+$ is such that $\phi(x) = \phi(-x)$ and
\begin{equation}
  \label{13A}
 c_\phi := \int_{\mathbb{R}^d} \left( 1 - e^{-\phi(x)}\right)d x < \infty.
\end{equation}
In the sequel, when we speak of the model we consider, we mean the one defined in  (\ref{C2}) -- (\ref{13A}).
We also call it continuum Kawasaki system.

The main reason for us to choose the rates as in (\ref{C3}) is that
any grand canonical Gibbs measure with potential $\phi$, see e.g., \cite{Ruelle}, is invariant (even symmetrizing) for the dynamics generated by (\ref{C2}) with such rates, see \cite{gibbs}.

As is usual for Markov dynamics, the `generator' (\ref{C2}) enters the
backward Kolmogorov equation
\begin{equation}
  \label{R2}
\frac{d}{dt} F_t = L F_t , \qquad F_t|_{t=0} = F_0,
\end{equation}
where, for each $t$, $F_t$ is an observable.
In the  approach we follow, the states of the system are probability measures on
$\Gamma$, and hence $
\int_\Gamma F d \mu $
can be considered as the value of observable $F$ in state $\mu$. This pairing allows one to define also the corresponding forward Kolmogorov or Fokker-Planck equation
\begin{equation}
  \label{R1}
\frac{d}{dt} \mu_t = L^* \mu_t, \qquad \mu_t|_{t=0} = \mu_0.
\end{equation}
The evolutions described by (\ref{C2}) and (\ref{R1}) are mutually  dual in the sense that
\begin{equation*}
\int_{\Gamma} F_t d \mu_0 = \int_{\Gamma} F_0 d \mu_t.
\end{equation*}
Thus, the Cauchy problem in  (\ref{R1}) determines the evolution of states of our model. If we were able
to solve it for all possible probability measures as initial conditions, we could construct
a Markov process on $\Gamma$. For nontrivial models, however, including that considered in this work, this is far beyond the possibilities of
the available technical tools. The main reason for this is that the configuration space $\Gamma$ has a complex topological structure. Furthermore, the mere existence of the process related to (\ref{R1}) would
 not be enough for drawing conclusions on the collective behavior of the considered system.
The basic idea of the approach which we follow is to solve (\ref{R1}) not for all possible $\mu_0$, but only for
those belonging to a properly chosen class of probability measures on $\Gamma$. It turns out that even with such restrictions the direct solving (\ref{R1}) is also unattainable, at least so far.
Then the solution in question is obtained by employing the so called \textit{moment} or \textit{correlation} functions. Similarly as a probability measure on $\mathbb{R}$ is characterized by its moments, a probability measure on $\Gamma$ can
be characterized by its correlation functions. Of course, as not every measure on $\mathbb{R}$ has all moments, not every  measure on $\Gamma$ possesses correlation functions. The mentioned  restriction in the choice of $\mu_0$ takes into account, among others, also this issue.

By certain combinatoric calculations, one transforms (\ref{R1}) into the following Cauchy problem
\begin{equation}
  \label{R4}
\frac{d}{dt} k_t = L^\Delta k_t, \qquad k_t|_{t=0} = k_0,
\end{equation}
where $k_0$ is the correlation function of $\mu_0$. Note that the equation in (\ref{R4}) is, in fact, an infinite chain of coupled linear equations. Then the construction of the evolution of states $\mu_0 \mapsto \mu_t$
is performed by: (a) solving (\ref{R4}) with $k_0 = k_{\mu_0}$; (b) proving that, for each $t$, there exists a \textit{unique} probability measure $\mu_t$ such that $k_t = k_{\mu_t}$. This way of constructing the evolution of states is, in a sense, analogous to that
suggested by  N. N. Bogoliubov \cite{Bog} in the statistical approach to the Hamiltonian dynamics of large systems of interacting physical particles,
cf. \cite{I,DF,L} and also a review in \cite{Dob}.  In the theory of such systems, the equation analogous to (\ref{R4}) is called  \textit{BBGKY chain} \cite{Dob}.

 The description based on (\ref{R1}) or (\ref{R4}) is \textit{microscopic} since one deals with coordinates of individual particles; cf. the Introduction in \cite{P}. More coarse-grained levels are \textit{meso- and macroscopic} ones. They are attained by appropriate space and time scalings \cite{P,spohn}. Of course, certain details of the system's behavior are then lost.
Kinetic equations provide a space-dependent mean-field-like approximate description of the evolution of infinite particle systems. For systems of physical particles, such an equation  is the Boltzmann equation
related to the BBGKY chain, cf. Section 6 in \cite{Dob} and also \cite{P,spohn}.
Nowadays, a mathematically consistent way of constructing the mesoscopic description based on kinetic equations is the procedure analogous to the Vlasov scaling in plasma physics, see \cite{DimaN}.
In its framework, we obtain from (\ref{R4}) a new chain of linear equations for limiting `correlation functions' $r_t$, called \textit{Vlasov hierarchy}.
Note that these $r_t$ may not be correlation functions at all but they  have one important property. Namely, if the initial state $\mu_0$
is the Poisson measure with density $\varrho_0 $,
then $r_t$ is the correlation function for  the Poisson measure $\mu_t$ with the density $\varrho_t $ which solves the corresponding kinetic equation.

In the present article, we aim at:
\begin{itemize}
  \item constructing the evolution of states $\mu_0 \mapsto \mu_t$ of the model (\ref{C2}), (\ref{C3}) by solving (\ref{R4}) and then by identifying its solution with a unique $\mu_t$;
  \item deriving rigorously the limiting Vlasov hierarchy, which includes also the convergence of rescaled correlation functions to the limiting functions $r_t$, as well as deriving the kinetic equation;
  \item studying the solvability of the kinetic equation.
\end{itemize}
Let us make some comments.
When speaking of the evolution of
states, one might
distinguish between equilibrium and non-equilibrium cases.
The
equilibrium evolution is built with the help of the reversible
 measures, if such exist for the considered model, and with the corresponding
Dirichlet forms. Recall that, for the choice as in (\ref{C3}), such reversible measures are grand canonical Gibbs measures.
The result is a stationary Markov process, see \cite{gibbs} where a version of the model studied in this work was considered.
Note that in this framework, the evolution is restricted to the set of states which are absolutely continuous with respect to
the corresponding Gibbs measures.
The non-equilibrium evolution, where initial states
can be ``far away" from equilibrium, is much more
interesting and much more complex  -- for the model considered in
this work, it has been constructed for noninteracting particles only, see
\cite{free}. In this article, we go further in this direction and construct the non-equilibrium evolution for the continuum Kawasaki system with repulsion. Results similar to those presented here were
obtained for a continuum Glauber model in \cite{FKK}, and for a spatial ecological model in \cite{FKKK}.

There exists  a rich theory of interacting particle systems  based on continuous time Markov
processes, which studies so called
lattice models, see \cite{ligget} and Part II of \cite{spohn}, and also \cite{LiggettP} for the latest results.
The essential common feature of these models is that the particles are distributed over a discrete set (lattice), typically $\mathbb{Z}^d$.
However, in many real-world applications, such as population biology or spatial ecology, the habitat, i.e., the space where the particles are placed, should essentially be continuous, cf. \cite{Neuhauser}, which we take into account in this work.
 In statistical physics, a lattice model of `hopping spins'  was put forward in \cite{kawaa}, see also a review in  \cite{kawaa1}. There exists an extended theory of interacting particles hopping over $\mathbb{Z}^d$, cf. \cite[Section 1 in Part II]{spohn}, and also,
e.g., \cite{Bovier,dH} for some aspects of the recent development. However, this theory cannot be applied to continuum Kawasaki systems
for a number of reasons. One of which is that a
bounded $K\subset \mathbb{R}^d$ can contain \textit{arbitrary} number of particles, whereas the number of
particles contained in a bounded $K\subset \mathbb{Z}^d$ is at most $|K|$.

\subsection{The overview of the results}



The microscopic description is performed in Section \ref{Sec3}
 in two steps.
\underline{First}, we prove that, for a given correlation function $k_0$, the  problem (\ref{R4}) has a unique classical solution $k_t$
on a bounded time interval $[0,T)$ and in a Banach space, somewhat bigger than that  containing $k_0$, cf. Theorems \ref{1tm} and \ref{1atm}.
Here bigger means that the initial
space is a proper subspace of the latter. The parameter $T>0$ is related to the `difference' between the spaces.
The main characteristic feature of both Banach spaces is that if their elements are correlation functions of some probability measures on $\Gamma$, then these measures are sub-Poissonian, cf. Definition \ref{SPdf} and Remark \ref{SPrk}. This latter property is important in view of the mesoscopic description which we construct subsequently, cf. Remark \ref{SP1rk}. The restriction of the evolution $k_0 \mapsto k_t$ to a bounded time interval is because we failed to apply to (\ref{R4})  semigroup methods, or similar techniques,  which would allow for solving this equation on $[0,+\infty)$  in the mentioned Banach spaces. Our method is based on Ovcyannikov's observation, cf.   \cite[pp. 9--13]{[D]} and \cite{O}, that an unbounded operator can be redefined as a bounded one acting, however, from a `smaller' to a `bigger' space, both belonging to a scale of Banach spaces, indexed by $\vartheta\in\mathbb{R}$.
The essential fact here is that the norm of such a bounded operator has an upper bound proportional to $(\vartheta'' - \vartheta')^{-1}$, see (\ref{C9}).
This implies that the expansion for $k_t$ in powers of $t$ converges for $t\in [0,T)$, cf. (\ref{8CD}) and (\ref{Lest}).
\underline{Second},  we prove that the evolution $k_0 \mapsto k_t$ corresponds to the evolution $\mu_0 \mapsto \mu_t$ of \textit{uniquely} determined probability measures
on $\Gamma$ in the following sense. In Theorem \ref{Qtm}, we show that if $k_0$ is the correlation function of a sub-Poissonian measure $\mu_0$,
then, for each $t\in (0,T)$, $k_t$ is also a correlation function for a \textit{unique} sub-Poissonian measure $\mu_t$.
The proof is based on the approximation of states of the infinite system by probability measures on $\Gamma$ supported on the set of finite configurations $\Gamma_0 \subset \Gamma$ (we call such measures $\Gamma_0$-states). The evolution of the latter states can be derived directly from (\ref{R1}), which we perform in Theorem \ref{0Ktm}.
It is described by a stochastic semigroup constructed
with the help of a version of Miyadera's theorem obtained in  \cite{[TV]}. Then we prove that the correlation functions of the
mentioned states supported on $\Gamma_0$ weakly converge to the solution $k_t$, which implies that it has the positivity property as in (\ref{Q1}), which by Proposition \ref{Tobpn}
yields that $k_t$ is also a correlation function for a unique state.

The mesoscopic description is performed in Section \ref{Sec4} in the framework of the scaling method developed in \cite{DimaN}.
First, we  derive an analog of (\ref{R4}) for the rescaled
correlation functions, that is, the Cauchy problem in (\ref{C12}). This problem contains the scaling parameter
$\varepsilon>0$, which is supposed to tend to zero in the mesoscopic limit. In this limit, we obtain another Cauchy problem, given in
(\ref{C22}). By the results of Section \ref{Sec3}, we readily prove the existence of classical solutions of both (\ref{C12}) and (\ref{C22}). The essence of the scaling technique which we use is that the evolution $r_0\mapsto r_t$ obtained from (\ref{C22}) preserves the set of
correlation functions of Poisson measures, cf. Lemma \ref{Vlm}. Then the density $\varrho_t$ that corresponds to $r_t$ satisfies the kinetic equation (\ref{C24}), which we then transform into an integral equation, cf. (\ref{C25}). For its eventual solutions,
by the Gronwall inequality we obtain an a priori bound, cf. (\ref{0C}), (\ref{0C1}), by means of which we prove the existence of a unique solution of both (\ref{C24}) and (\ref{C25}) on $[0,+\infty)$, which implies the global evolution $r_0\mapsto r_t$, cf. Theorem \ref{Vatm}.
Finally, in Theorem \ref{3tm} we show that the rescaled correlation functions converge to the Poisson correlation functions $r_t$
as $\varepsilon \to 0^+$, uniformly on compact subsets of $[0,T)$. This result links both micro-
and mesoscopic evolutions constructed in this work.

Let us mention some open problems related to the model studied in this work. The existence of the global mesoscopic evolution does not, however, imply
that the restriction of the microscopic evolution to a bounded time interval is only a technical problem. One cannot exclude that, due to an infinite number of jumps, $k_t$ finally leaves any space of the type of (\ref{2D}).
It is still unclear whether
 the global evolution $k_0 \mapsto k_t$ exists in any of Banach spaces reasonably bigger than those used in Theorems \ref{1tm} and \ref{1atm}.

A very interesting problem, in the spirit of the philosophy of \cite{Cox}, is to relate the rate of convergence in (\ref{0Q13}) to the value of $t$, which determines the space in which $k_t$ lies, cf. Theorem \ref{1atm}. Another open problem is the existence of globally bounded solutions of the kinetic equation (\ref{C24}). It can be proven that, for a local repulsion, this is the case. Namely, if $\phi$ in (\ref{C24}) is such that
$(\phi * \varrho)(x) = \varkappa \varrho(x)$ for all $x$ and some $\varkappa \geq 0$, and if $\varrho_0$ is a bounded continuous function on $\mathbb{R}^d$, then the solution $\varrho_t$ is also a continuous function, cf. Corollary \ref{o1rk}, such that $\varrho_t \leq \sup_{x\in \mathbb{R}^d}\varrho_0 (x) + \epsilon$ for all $t>0$ and any $\epsilon>0$. From this one can see how important can be the relation between the radii of the jump kernel $a$ and of the repulsion potential $\phi$.


\section{The basic notions }

In this paper, we work in the approach of
\cite{FKKK,Dima0,Dima,Dima1,Dima2,Tobi} where all the relevant
details can be found.

\subsection{The configuration spaces}
By $\mathcal{B}(\mathbb{R}^d)$ and $\mathcal{B}_{\rm
b}(\mathbb{R}^d)$ we denote the sets of all Borel and all
bounded Borel subsets of $\mathbb{R}^d$, respectively.
The configuration space $\Gamma$ is
\begin{equation*} 
 \Gamma=\{\gamma\subset\mathbb R^d :
 |\gamma\cap K|<\infty\text{ for any compact $K\subset\mathbb R^d$
 }\}.
\end{equation*}
Each
$\gamma \in \Gamma$ can be identified with the following positive
Radom measure
\[
\gamma(d x) = \sum_{y\in \gamma} \delta_y(d x) \in \mathcal{M}(\mathbb{R}^d),
\]
where $\delta_y$ is the Dirac measure centered at $y$, and
$\mathcal{M}(\mathbb{R}^d)$ denotes the  set of all positive Radon measures on $\mathcal{B}(\mathbb{R}^d)$.
This allows one to consider $\Gamma$ as a subset of $\mathcal{M}(\mathbb{R}^d)$, and hence to endow it
with the vague topology. The latter is the weakest topology in which all the maps
\[
\Gamma \ni \gamma \mapsto \int_{\mathbb{R}^d} f(x) \gamma(d x) = \sum_{x\in \gamma} f(x), \qquad f\in C_0(\mathbb{R}^d),
\]
are continuous. Here $C_0 (\mathbb{R}^d)$ stands for the set of all
continuous functions $f:\mathbb{R}^d \to \mathbb{R}$ which have compact support.
The vague topology on $\Gamma$ admits a metrization, which turns it into a complete and
separable (Polish) space, see, e.g., \cite[Theorem 3.5]{Oles0}.
By $\mathcal{B}(\Gamma)$ we denote the corresponding Borel $\sigma$-field.

For $n\in
\mathbb{N}_0 :=\mathbb{N}\cup \{0\}$, the set of $n$-particle configurations
in $\mathbb{R}^d$ is
\begin{equation}
\label{1U} \Gamma^{(0)} = \{ \emptyset\}, \qquad  \Gamma^{(n)} =
\{\eta \subset \mathbb{R}^d: |\eta| = n \}, \ \ n\in \mathbb{N}.
\end{equation}
For $n\geq 2$,
$\Gamma^{(n)}$ can be identified with the symmetrization of the
set $$\{(x_1, \dots , x_n)\in (\mathbb{R}^d)^n: x_i \neq x_j, \ {\rm for} \ i\neq
j\}\subset (\mathbb{R}^{d})^n,$$ which allows one to introduce the
corresponding topology and hence the Borel $\sigma$-field
$\mathcal{B}(\Gamma^{(n)})$. The set of finite configurations is
\begin{equation}
  \label{2U}
\Gamma_{0} := \bigsqcup_{n\in \mathbb{N}_0}  \Gamma^{(n)}.
\end{equation}
We equip it with the topology of the disjoint union and hence with
the Borel $\sigma$-field $\mathcal{B}(\Gamma_{0})$. Obviously, $\Gamma_0$ is a subset of
$\Gamma$, cf. (\ref{C100}). However, the topology just mentioned and that induced from $\Gamma$ do not coincide.
At the same time, $\Gamma_0 \in \mathcal{B}(\Gamma)$.
In the sequel, by $\Lambda$ we denote
a bounded subset of $\mathbb{R}^d$, that is, we always mean $\Lambda \in \mathcal{B}_{\rm b} (\mathbb{R}^d)$.
For such $\Lambda$, we set
\[
 \Gamma_\Lambda = \{ \gamma \in \Gamma: \gamma \subset \Lambda\}.
\]
Clearly, $\Gamma_\Lambda$ is also a measurable subset of $\Gamma_0$ and the following holds
\[
\Gamma_\Lambda =  \bigsqcup_{n\in \mathbb{N}_0} \Bigl(\Gamma^{(n)} \cap \Gamma_\Lambda\Bigr),
\]
which allows one to equip $\Gamma_\Lambda$ with the topology induced by that of $\Gamma_0$. Let
$\mathcal{B}(\Gamma_\Lambda)$ be the corresponding Borel $\sigma$-field.
It is clear that, for $A \in \mathcal{B}(\Gamma_0)$, $\Gamma_\Lambda \cap A \in \mathcal{B}(\Gamma_\Lambda)$.
It can be proven, see Lemma~1.1 and Proposition~1.3 in \cite{Obata}, that
\begin{equation}
 \label{K2}
\mathcal{B}(\Gamma_\Lambda) = \{ \Gamma_\Lambda \cap A : A \in \mathcal{B}(\Gamma)\},
\end{equation}
and hence
\begin{equation}
  \label{0K2}
  \mathcal{B}(\Gamma_0) = \{ A \in \mathcal{B}(\Gamma): A \subset \Gamma_0\}.
\end{equation}
Next, we define the projection
\begin{equation}
 \label{K3}
\Gamma \ni \gamma \mapsto p_\Lambda (\gamma)= \gamma_\Lambda := \gamma\cap \Lambda, \qquad \Lambda \in \mathcal{B}_{\rm b} (\mathbb{R}^d).
\end{equation}
It is known \cite[p. 451]{AlKonR} that $\mathcal{B}(\Gamma)$ is the smallest $\sigma$-algebra of subsets of
$\Gamma$ such that the maps $p_\Lambda$ with all $\Lambda \in \mathcal{B}_{\rm b} (\mathbb{R}^d)$ are
$\mathcal{B}(\Gamma)/\mathcal{B}(\Gamma_\Lambda)$ measurable. This means that $(\Gamma, \mathcal{B}(\Gamma))$ is the projective limit
of the measurable spaces $(\Gamma_\Lambda, \mathcal{B}(\Gamma_\Lambda))$, $\Lambda \in \mathcal{B}_{\rm b} (\mathbb{R}^d)$.
A set $A\in \mathcal{B}(\Gamma_0)$ is said to be bounded if
\begin{equation}
 \label{6A}
 A \subset \bigsqcup_{n=0}^N \Gamma^{(n)}_\Lambda
\end{equation}
for some $\Lambda \in \mathcal{B}_{\rm b}(\mathbb{R}^d)$ and $N\in
\mathbb{N}$. The smallest $\Lambda$ such that $A \subset \Gamma_\Lambda$ will be called the \textit{support} of $A$.

\subsection{Measures and functions}

Given $n\in \mathbb{N}$, by $m^{(n)}$ we denote the restriction of
the Lebesgue product measure $dx_1 dx_2 \cdots dx_n $ to
$(\Gamma^{(n)}, \mathcal{B}(\Gamma^{(n)}))$. The
Lebesgue-Poisson measure with intensity $\varkappa>0$ is a measure on $(\Gamma_0, \mathcal{B}(\Gamma_0))$ defined by
\begin{equation}
  \label{1A}
  \lambda_\varkappa = \delta_{\emptyset} + \sum_{n=1}^\infty \frac{\varkappa^n}{n!} m^{(n)}.
\end{equation}
For $\Lambda\in \mathcal{B}_{\rm b}(\mathbb{R}^d)$, the restriction of
$\lambda_\varkappa$ to $\Gamma_{\Lambda}$ will be denoted by $\lambda_\varkappa^\Lambda$.
This is a finite measure on $\mathcal{B}(\Gamma_\Lambda)$ such that
\[
\lambda^\Lambda_\varkappa (\Gamma_\Lambda) = \exp[\varkappa m(\Lambda)],
\]
where $m(\Lambda):=m^{(1)}(\Lambda)$ is the Lebesgue measure of $\Lambda$. Then
\begin{equation}
 \label{3A}
 \pi^\Lambda_\varkappa  := \exp ( - \varkappa m(\Lambda)) \lambda^\Lambda_\varkappa
\end{equation}
is a probability measure on $\mathcal{B}(\Gamma_\Lambda)$.
It can be shown \cite{AlKonR} that the family $\{\pi^\Lambda_\varkappa\}_{\Lambda \in \mathcal{B}_{\rm b}(\mathbb{R}^d)}$
is consistent, and hence there exists a unique probability measure, $\pi_\varkappa$, on $\mathcal{B}(\Gamma)$ such that
\[
 \pi^\Lambda_\varkappa = \pi_\varkappa \circ p^{-1}_\Lambda, \qquad \Lambda \in \mathcal{B}_{\rm b} (\mathbb{R}^d),
\]
where $p_\Lambda$ is the same as in (\ref{K3}).  This $\pi_\varkappa$ is called the Poisson measure.
 The Poisson
measure $\pi_{\varrho}$ corresponding  to the density $\varrho:
\mathbb{R}\rightarrow \mathbb{R}_{+}$ is introduced by means of the
measure $\lambda_{\varrho}$, defined as in (\ref{1A}) with $\varkappa m$
replaced by $m_\varrho$, where, for $\Lambda \in \mathcal{B}_{\rm
b}(\mathbb{R}^d)$,
\begin{equation}
  \label{3Aa}
m_{\varrho} (\Lambda) := \int_{\Lambda} \varrho (x) d x,
\end{equation}
which is supposed to be finite. Then $\pi_{\varrho}$ is defined by its projections
\begin{equation}
  \label{3Ab}
 \pi^\Lambda_{\varrho} = \exp ( - m_{\varrho}(\Lambda)) \lambda^\Lambda_{\varrho}.
\end{equation}
For $\varkappa =1$, we shall drop the subscript and consider the Lebesgue-Poisson measure $\lambda$ and the Poisson measure $\pi$.

For a measurable $f:\mathbb{R}^d \rightarrow \mathbb{R}$ and $\eta
\in \Gamma_0$, the Lebesgue-Poisson exponent is
\begin{equation}
  \label{4A}
  e(f, \eta) = \prod_{x\in \eta} f(x), \qquad e(f, \emptyset ) = 1.
\end{equation}
Clearly, $ e(f, \cdot)\in L^1 (\Gamma_0, d \lambda)$ for any $f \in L^1
(\mathbb{R}^d):= L^1
(\mathbb{R}^d, dx)$, and
\begin{equation}
  \label{5A}
\int_{\Gamma_0}    e(f, \eta) \lambda (d\eta) =
\exp\left\{\int_{\mathbb{R}^d} f(x) d x \right\}.
\end{equation}
 By $B_{\rm bs} (\Gamma_0)$ we denote the set of all
bounded measurable functions $G: \Gamma_0 \rightarrow \mathbb{R}$,
which have bounded  supports. That is, each such $G$ is the zero
function on $\Gamma_0 \setminus A$ for some bounded $A$, cf. (\ref{6A}).  Note that
any measurable  $G: \Gamma_0 \rightarrow \mathbb{R}$ is in fact a
sequence of measurable symmetric functions $G^{(n)} :
(\mathbb{R}^d)^n  \rightarrow \mathbb{R}$ such that, for $\eta = \{x_1 , \dots, x_n\}$,
$G(\eta) = G^{(n)}(x_1 , \dots, x_n)$. We say that $F:\Gamma\to \mathbb{R}$ is a \textit{cylinder} function
if there exists $\Lambda \in \mathcal{B}_{\rm b} (\mathbb{R}^d)$ and $G:\Gamma_\Lambda \to \mathbb{R}$ such that
$F(\gamma) = G(\gamma_\Lambda)$ for all $\gamma\in \Gamma$. By $\mathcal F_{\text{cyl}}(\Gamma)$ we denote the set of all
measurable cylinder functions.
 For $\gamma \in \Gamma$, by
writing $\eta \Subset  \gamma$ we mean that $\eta \subset \gamma$
and $\eta$ is finite, i.e., $\eta \in \Gamma_0$. For $G \in B_{\rm
bs}(\Gamma_0)$, we set
\begin{equation}
  \label{7A}
  (KG)(\gamma) = \sum_{\eta \Subset  \gamma} G(\eta), \qquad \gamma \in \Gamma.
\end{equation}
Clearly  $K$ maps $B_{\rm bs}(\Gamma_0)$ into
$\mathcal{F}_{\rm cyl}(\Gamma)$ and is linear and positivity
preserving. This map plays an important role in the theory of
configuration spaces, cf. \cite{Tobi}.

By $\mathcal{M}^1 (\Gamma)$ we denote the set of all
probability measures on $(\Gamma, \mathcal{B}(\Gamma))$, and let
 $\mathcal{M}^1_{\rm fm} (\Gamma)$ denote the subset of $\mathcal{M}^1 (\Gamma)$ consisting of all
measures which have
finite local moments, that is, for which
\begin{equation*}
\int_\Gamma |\gamma_\Lambda|^n \mu(d \gamma) < \infty
\quad \ \ \ {\rm for} \ \ {\rm all} \ \ n\in \mathbb{N} \ \ {\rm
and}\ \ \Lambda \in \mathcal{B}_{\rm b} (\mathbb{R}^d).
\end{equation*}
\begin{definition}
 \label{LACdf}
A measure $\mu \in
\mathcal{M}^1_{\rm fm} (\Gamma)$ is said to be {\it locally
absolutely continuous} with respect to the Poisson measure $\pi$ if,
for every $\Lambda \in \mathcal{B}_{\rm b} (\mathbb{R}^d)$, the projection
\begin{equation}
 \label{Lambda}
\mu^\Lambda := \mu \circ p_\Lambda^{-1}
\end{equation}
is absolutely continuous
with respect to $\pi^\Lambda$ and hence with respect to $\lambda^\Lambda$, see (\ref{3A}).
\end{definition}
A measure $\chi$ on $(\Gamma_0, \mathcal{B}(\Gamma_0))$ is said to
be {\it locally finite} if $\chi(A)< \infty$ for every bounded measurable
$A\subset \Gamma_0$. By $\mathcal{M}_{\rm lf} (\Gamma_0)$ we denote
the set of all such measures.
Let a measurable $A\subset \Gamma_0$ be bounded, and let $\mathbb{I}_A$ be its
indicator function on $\Gamma_0$. Then $\mathbb{I}_A$ is in $B_{\rm
bs}(\Gamma_0)$, and hence one can apply (\ref{7A}). For  $\mu \in
\mathcal{M}^1_{\rm fm} (\Gamma)$, we let
\begin{equation}
  \label{9A}
  \chi_\mu (A) = \int_{\Gamma} (K\mathbb{I}_A) (\gamma) \mu(d \gamma),
\end{equation}
which uniquely determines a measure $\chi_\mu \in \mathcal{M}_{\rm
lf}(\Gamma_0))$. It is called the {\it correlation measure} for
$\mu$. For instance, let $A\subset \Gamma^{(n)}\subset\Gamma_0$ and let
$\Lambda\in \mathcal{B}_{\rm b}(\mathbb{R}^d)$ be the support of $A$, cf. (\ref{6A}).
Then  $(K\mathbb I_{A})(\gamma)$ is the number of distinct
$n$-particle sub-configurations of $\gamma$ contained in $\Lambda$, and thus
 $\chi_\mu (A)$ is the expected number of such sub-configurations in $\Lambda$ in state $\mu$.
In particular, if $A\subset \Gamma^{(1)}$, then $\chi_\mu (A)$ is just the
expected number of particles in state $\mu$ contained in $\Lambda$.

The equation \eqref{9A} defines a map $K^*: \mathcal{M}^1_{\rm fm} (\Gamma)
\rightarrow \mathcal{M}_{\rm lf}(\Gamma_0))$ such that $K^*\mu =
\chi_\mu$. In particular, $K^*\pi = \lambda$. It is known, see
\cite[Proposition 4.14]{Tobi}, that $\chi_\mu $ is absolutely
continuous with respect to $\lambda$ if $\mu$ is locally absolutely
continuous with respect to $\pi$. In this case, we have that, for any
$\Lambda \in \mathcal{B}_{\rm b} (\mathbb{R}^d)$ and $\lambda^\Lambda$-almost all $\eta \in \Gamma_\Lambda$,
\begin{eqnarray}
  \label{9AA}
k_\mu (\eta)  =  \frac{d \chi_\mu}{d \lambda}(\eta) & = &
\int_{\Gamma_{\Lambda}} \frac{d\mu^\Lambda}{d \pi^\Lambda} (\eta
\cup \gamma) \pi^\Lambda (d \gamma)\\[.2cm] & = & \int_{\Gamma_{\Lambda}} \frac{d\mu^\Lambda}{d \lambda^\Lambda} (\eta
\cup \xi) \lambda^\Lambda (d \xi) \nonumber.
\end{eqnarray}
The Radon-Nikodym derivative $k_\mu$ is called the {\it correlation
function} corresponding to the measure $\mu$.
As all real-valued measurable functions on $\Gamma_0$, each $k_\mu$ is the collection of measurable $k_\mu^{(n)}: (\mathbb{R}^d)^n \to \mathbb{R}$ such that
$k_\mu^{(0)} \equiv 1$, and $k_\mu^{(n)}$, $n\geq 2$, are symmetric. In particular, $k_\mu^{(1)}$ is the particle's density in state $\mu$, cf. (\ref{9A}).

Recall that by $B_{\rm bs} (\Gamma_0)$ we denote the set of all
bounded measurable functions $G: \Gamma_0 \rightarrow \mathbb{R}$,
which have bounded  supports.
We also set
\begin{equation}
  \label{bbs}
 B_{\rm bs}^+ (\Gamma_0)=\{ G \in B_{\rm bs}(\Gamma_0): (KG)(\gamma) \geq 0\}.
\end{equation}
The following fact is known, see Theorems 6.1, 6.2 and Remark 6.3  in \cite{Tobi}.
\begin{proposition}
 \label{Tobpn}
Suppose
$\chi \in \mathcal{M}_{\rm lf}(\Gamma_0))$ has the properties
\begin{equation}
  \label{10A}
\chi(\{\emptyset\}) = 1 , \qquad \int_{\Gamma_0} G (\eta ) \chi(d \eta) \geq 0 ,
\end{equation}
for each $G\in B^+_{\rm bs}(\Gamma_0)$. Then
there exist $\mu\in \mathcal{M}^1_{\rm fm}(\Gamma)$ such that $K^*
\mu =\chi$. For the uniqueness of such $\mu$, it is enough that the Radon-Nikodym derivative
(\ref{9AA}) of
$\chi$ obeys
\begin{equation}
  \label{11A}
k(\eta) \leq  \prod_{x \in \eta} C_R (x) ,
\end{equation}
for all $\eta\in \Gamma_0$ and for some
locally integrable $C_R : \mathbb{R}^d \rightarrow \mathbb{R}_{+}$.
\end{proposition}
Let $\pi_\varrho$ be the Poisson measure as in (\ref{3Ab}), and let $\Lambda$ be the support of a given
bounded $A\subset \Gamma_0$. If $A\subset \Gamma^{(n)}$, cf. (\ref{1U}), then
\begin{equation}
  \label{chi}
\chi_{\pi_{\varrho}} (A) = \int_{\Lambda^n} \varrho(x_1) \cdots \varrho(x_n) dx_1 \cdots dx_n = \left(\int_{\Lambda} \varrho(x) dx \right)^n,
\end{equation}
which, in particular, means that particles appear in $\Lambda$ independently. In  this case,
\begin{equation}
  \label{11Aa}
 k_{\pi_\varrho} (\eta) = e(\varrho, \eta),
\end{equation}
where $e$ is as in (\ref{4A}). In particular,
\begin{equation}
  \label{11Ab}
 k^{(n)}_{\pi_\varrho} (x_1 ,\dots, x_n) = \varrho(x_1) \cdots \varrho(x_n), \qquad n\in \mathbb{N}.
\end{equation}
\begin{definition}
  \label{SPdf}
A locally
absolutely continuous measure $\mu \in
\mathcal{M}^1_{\rm fm} (\Gamma)$, cf. Definition \ref{LACdf}, is called sub-Poissonian if its correlation function $k_\mu$ obeys (\ref{11A})
for some locally integrable $C_R : \mathbb{R}^d \rightarrow \mathbb{R}_{+}$.
\end{definition}
\begin{remark}
  \label{SPrk}
If $\mu$ is sub-Poissonian and $A$ is as in (\ref{chi}), then
$$\chi_\mu(A) \leq C^n \left( \int_\Lambda k_\mu^{(1)} (x) dx\right)^n,$$ for some $C>0$.
That is, the correlation measure is controlled by the density in this case. For instance, if one knows that
$k_{\mu_t}^{(1)}$ does not explode for all $t>0$, then so does $k_{\mu_t}$, and hence $\mu_t$ exists for all $t>0$.
A faster increase of $\chi_\mu(A)$, e.g., as $n!$, can be interpreted as \textit{clustering} in state $\mu$.
\end{remark}

Finally, we present the following integration rule, cf.
\cite[ Lemma 2.1]{Dima0},
\begin{equation}
  \label{12A}
\int_{\Gamma_0} \sum_{\xi \subset \eta} H(\xi, \eta \setminus \xi,
\eta) \lambda (d \eta) =   \int_{\Gamma_0}\int_{\Gamma_0}H(\xi,
\eta, \eta\cup \xi) \lambda (d \xi) \lambda (d \eta),
\end{equation}
 which holds for any appropriate function $H$.

\section{Microscopic dynamics}
\label{Sec3}
In view of the fact that $\Gamma$ contains also infinite configurations, the direct construction of the evolution
based on (\ref{R2}) and (\ref{C2}) cannot be done, and thus we pass to
the description based on correlation functions, cf. (\ref{R4}) and (\ref{9AA}).
The `generator' in (\ref{R4}) has the form, cf. \cite[eq. (4.8)]{Dima2}
\begin{eqnarray}
  \label{C5c}
\left( {L}^\Delta k\right) (\eta) & = & \sum_{y\in \eta} \int_{\mathbb{R}^d} a (x-y)
 e( \tau_y, \eta \setminus y \cup x)\\[.2cm] & \times & \left( \int_{\Gamma_0}e(t_y, \xi) k( \xi \cup x \cup \eta \setminus y)
 \lambda (d \xi) \right) dx \nonumber \\[.2cm] & - & \int_{\Gamma_0} k(\xi \cup \eta)
  \left( \sum_{x\in \eta} \int_{\mathbb{R}^d} a(x-y) e(\tau_y,\eta) \right.\nonumber\\[.2cm]
  & \times &   e(t_y, \xi) d y \bigg{)} \lambda (d\xi).  \nonumber
\end{eqnarray}
Here $e$ is as in (\ref{4A}) and
\begin{equation}
  \label{18A}
  t_x (y) = e^{-\phi (x-y)} - 1, \qquad \tau_x (y) = t_x (y) + 1.
\end{equation}
We shall also consider the following auxiliary evolution $G_0 \mapsto G_t$, dual to that $k_0 \mapsto k_t$ described by (\ref{R4}) and (\ref{C5c}).
The duality is understood in the sense
\begin{equation}
  \label{C5d}
\langle \! \langle G_t , k_0 \rangle \! \rangle = \langle \! \langle G_0 ,
k_t \rangle \!\rangle,
\end{equation}
where
\begin{equation}
  \label{19A}
  \langle \! \langle G,k\rangle \! \rangle := \int_{\Gamma_0} G(\eta) k(\eta)\lambda (d \eta) ,
\end{equation}
and $\lambda$ is the Lebesgue-Poisson measure defined in (\ref{1A}) with $\varkappa=1$.
The `generators' are related to each other by
\begin{equation}
  \label{dualL}
 \langle \! \langle G , L^\Delta k \rangle \! \rangle = \langle \! \langle \widehat{L} G,
k \rangle \!\rangle.
\end{equation}
Then the equation dual to (\ref{R4}) is
\begin{equation}
  \label{16A}
\frac{d}{dt} G_t = \widehat{L} G_t, \qquad G_t|_{t=0} = G_0,
\end{equation}
with, cf. \cite[eq. (4.7)]{Dima2},
\begin{eqnarray}
  \label{C5b}
\left( \widehat{L}G\right) (\eta) & = & \sum_{\xi \subset \eta} \sum_{x\in \xi}
\int_{\mathbb{R}^d} a (x-y)   e(\tau_y, \xi) \\[.2cm] & \times & e(t_y , \eta \setminus \xi)
 \left[ G(\xi \setminus x \cup y) - G(\xi) \right]dy. \nonumber
\end{eqnarray}
\subsection{The evolution of correlation functions}
\label{Sec3.1}
We consider (\ref{R4}) with $L^\Delta$ given in (\ref{C5c}). To place this problem in the right
context we introduce the following Banach spaces. Recall that a function $G:\Gamma_0 \to \mathbb{R}$ is
a sequence of $G^{(n)}: (\mathbb{R}^d)^n \to \mathbb{R}$, $n\in \mathbb{N}_0$, where $G^{(0)}$ is constant and
all $G^{(n)}$, $n\geq 2$, are symmetric. Let $k:\Gamma_0\to \mathbb{R}$ be such that $k^{(n)}\in L^\infty ((\mathbb{R}^d)^n)$, for $n\in \mathbb{N}$. For this $k$ and $\vartheta \in \mathbb{R}$, we set
\begin{equation}
 \label{o5}
\|k\|_\vartheta =  \sup_{n\in \mathbb{N}_0} \nu_n(k) \exp(\vartheta
n),
\end{equation}
where
\begin{equation}
  \label{0o5}
\nu_0(k) = |k^{(0)}|, \qquad \nu_n(k) = \|k^{(n)} \|_{L^\infty ((\mathbb{R}^d)^n)}, \quad n\in \mathbb{N}.
\end{equation}
Then
\begin{equation}
  \label{2D}
  \mathcal{K}_\vartheta := \{ k: \Gamma_0 \rightarrow \mathbb{R} : \|k\|_\vartheta < \infty\},
\end{equation}
is a real Banach space with norm (\ref{o5}) and usual point-wise linear operations. Note that
$\{\mathcal{K}_\vartheta:\vartheta \in \mathbb{R}\}$ is a scale of Banach spaces in the sense that
\begin{equation}
  \label{0o6}
   \mathcal{K}_\vartheta \subset   \mathcal{K}_{\vartheta'}, \qquad {\rm for} \ \ \vartheta > \vartheta'.
\end{equation}
As usual, by a classical solution of (\ref{R4}) in $\mathcal{K}_\vartheta$ on
time interval $I$, we understand a map $I\ni t\mapsto k_t\in\mathcal{K}_\vartheta$,
which is continuous on $I$, continuously differentiable on the interior of $I$, lies in the domain of $L^\Delta$, and
solves  (\ref{R4}).  Recall that we suppose
(\ref{13A}) and (\ref{C4}).
\begin{theorem}
 \label{1tm}
Given $\vartheta \in \mathbb{R}$ and $T>0$, we let
\begin{equation}\label{beta}
 \vartheta_0=\vartheta+  2 \alpha T \exp(c_{\phi}e^{-\vartheta}).
\end{equation}
Then the problem (\ref{R4}) with $k_0 \in
{\mathcal{K}}_{\vartheta_0}$ has a unique classical solution $k_t\in
{\mathcal{K}}_{\vartheta}$ on $[0,T)$.
\end{theorem}
According to the above theorem, for arbitrary $T>0$ and $\vartheta$, one can pick
the initial space such that the evolution $k_0 \mapsto k_t$ lasts in $\mathcal{K}_\vartheta$
until  $t<T$. On the other hand, if the initial space is given, the evolution
is restricted in time to the interval $[0,T(\vartheta))$ with
\begin{equation}
 \label{8}
T(\vartheta) = \frac{\vartheta_0 - \vartheta}{2 \alpha} \exp\left( - c_\phi e^{-\vartheta}\right).
\end{equation}
Clearly, $T(\vartheta_0) = 0$ and $T(\vartheta) \rightarrow 0$ as $\vartheta \rightarrow - \infty$.
Hence, there exists $T_*=T_* (\vartheta_0, \alpha, c_\phi)$ such that
$T(\vartheta) \leq T_*$ for all $\vartheta \in (-\infty, \vartheta_0]$.  Set
\begin{equation}
  \label{T}
 \vartheta(t)  =  \sup\{ \vartheta \in (-\infty , \vartheta_0]: t < T(\vartheta)\}.
\end{equation}
Then the alternative version
of the above theorem can be formulated as follows.
\begin{theorem}
 \label{1atm}
For every $\vartheta_0\in \mathbb{R}$, there exists $T_* = T_* (\vartheta_0, \alpha, c_\phi)$ such that
the problem (\ref{R4}) with $k_0 \in \mathcal{K}_{\vartheta_0}$ has a unique classical solution $k_t\in \mathcal{K}_{\vartheta(t)}$
on $[0, T_*)$.\end{theorem}
\vskip.1cm \noindent
{\it Proof of Theorem \ref{1tm}.}
Let $\vartheta \in \mathbb{R}$ be fixed. Set
\begin{equation}
  \label{DomL}
{\rm Dom} (L^\Delta) = \{ k \in \mathcal{K}_\vartheta: L^\Delta k \in \mathcal{K}_\vartheta\}.
\end{equation}
Given $k$, let $L^\Delta_1 k$ and $L^\Delta_2 k$
denote the first and the second summands in (\ref{C5c}), respectively. Then,
for  $\vartheta \leq \vartheta' <
\vartheta''$
and $k\in\mathcal K_{\vartheta''}$, we have
\begin{eqnarray*}
& &|(L^{\Delta}_1k)(\eta)|e^{\vartheta'|\eta|}\\[.2cm] & &  \quad  \le
\sum_{y\in\eta}\int_{\mathbb R^d}dx\text{ }a(x-y)\int_{\Gamma_0}\lambda(d\xi)|k(\eta\backslash y\cup x\cup\xi)|
\exp(\vartheta'|\eta\cup\xi|) \\[.2cm]
& &\quad\times \exp(-\vartheta''|\eta\cup\xi|)e(|t_y|,\xi)|e^{\vartheta'|\eta|}\\[.2cm]
& & \quad  \le\sum_{y\in\eta}\int_{\mathbb R^d}dx\text{ }a(x-y)\int_{\Gamma_0}\lambda(d\xi)\|k\|_{\vartheta''}e^{-\vartheta''|\xi|}\\[.2cm]
& &\quad\times e(|t_y|,\xi)|e^{-|\eta|(\vartheta''-\vartheta')}\\[.2cm]
& & \quad =\|k\|_{\vartheta''} \alpha \exp\big(e^{-\vartheta''}c_{\phi}\big)|\eta|e^{-|\eta|(\vartheta''-\vartheta'')}\\[.2cm]
& & \quad \le\|k\|_{\vartheta''} \alpha \exp\big( e^{-\vartheta''}c_{\phi}\big)\frac{1}{e(\vartheta''-\vartheta')},
\end{eqnarray*}
which holds for $\lambda$-almost all $\eta\in \Gamma_0$.
In the last line we used (\ref{5A}) and the following inequality
\begin{displaymath}
 \tau e^{-\delta \tau}\le {1}/{e\delta}, \qquad {\rm for} \ {\rm all} \ \ \tau\ge 0, \quad \delta >0.
\end{displaymath}
Similarly one estimates also $L^\Delta_2 k$, which finally yields
\begin{equation}
  \label{0o10}
\|L^\Delta k \|_{\vartheta'} \leq \frac{2\alpha}{e(\vartheta'' -
\vartheta')} \exp\left(c_\phi e^{-\vartheta''} \right)
\|k\|_{\vartheta''}, \qquad \vartheta''
> \vartheta',
\end{equation}
and hence, cf. (\ref{0o6}) and (\ref{DomL}),
\begin{equation}
  \label{0o7}
 {\rm Dom}(L^\Delta) \supset \mathcal{K}_{\vartheta'}, \qquad {\rm for} \ \ {\rm all} \ \  \vartheta' > \vartheta.
\end{equation}
By (\ref{0o10}), ${L}^\Delta$ can be defined as a bounded linear operator
${L}^\Delta:\mathcal{K}_{\vartheta''} \rightarrow
{\mathcal{K}}_{\vartheta'}$, $\vartheta' <
 \vartheta''$, with norm
\begin{equation}
  \label{C9}
  \|L^\Delta  \|_{\vartheta'' \vartheta'} \leq \frac{2\alpha}{e(\vartheta'' - \vartheta')}
\exp\left(c_\phi e^{-\vartheta''} \right).
\end{equation}
Given $k_0\in \mathcal{K}_{\vartheta_0}$, we seek the solution of (\ref{R4}) as the limit of the sequence
$\{k_{t,n}\}_{n\in \mathbb{N}_0 }\subset \mathcal{K}_\vartheta$,
where $k_{t,0} = k_0$ and
\begin{equation*}
k_{t,n} = k_0 + \int_0^t L^{\Delta} k_{s,n-1} d s, \qquad n \in \mathbb{N}.
\end{equation*}
The latter can be iterated to yield
\begin{equation}
 \label{8CD}
k_{t,n} = k_0 + \sum_{m=1}^n \frac{1}{m!}t^m
\big(L^{\Delta}\big)^m k_0.
\end{equation}
Then, for $n,p\in \mathbb{N}$, we have
\begin{equation}
  \label{0o11}
 \|k_{t,n} - k_{t,n+p}\|_\vartheta \leq  \sum_{m=n+1}^{n+p} \frac{t^m}{m!} \|(L^\Delta)^m\|_{\vartheta_0 \vartheta} \|k_0\|_{\vartheta_0}.
\end{equation}
For a given $m\in \mathbb{N}$ and $l = 0, \dots , m$, set
$\vartheta_l = \vartheta + (m- l )\epsilon$, $\epsilon = (\vartheta_0 -
\vartheta)/m$. Then by (\ref{C9}) and (\ref{8}), we get
\begin{eqnarray}
\label{Lest}
 \|(L^\Delta)^m\|_{\vartheta_0 \vartheta}  & \leq & \prod_{l=0}^m \|L^\Delta \|_{\vartheta_{l}\vartheta_{l+1}}  \leq \left(\frac{2 \alpha m\exp[c_\phi e^{-\vartheta}]}{e (\vartheta_0 - \vartheta)} \right)^m\\[.2cm]
 & = & \left( \frac{m}{e}\right)^m \frac{1}{[T(\vartheta)]^m}. \nonumber
\end{eqnarray}
Applying the latter estimate in (\ref{0o11}) we obtain that the sequence
$\{k^{(n)}_t\}_{n\in \mathbb{N}}$ converges in
$\mathcal{K}_\vartheta$ for $|t|< T$, and hence is differentiable, even real analytic, in $\mathcal{K}_\vartheta$ on the latter set.
From the proof above we see that $\{k^{(n)}_t\}_{n\in \mathbb{N}}$ converges
also in $\mathcal{K}_{\vartheta(t)}$ with $\vartheta(t)$ as in (\ref{T}), which proves Theorem \ref{1atm}.
The latter by (\ref{0o7}) yields  that $k_t \in {\rm Dom}(L^\Delta)$ for all $t< T(\vartheta)$, which completes the proof of the existence.
The uniqueness readily follows by the analyticity just mentioned.
$\square$ \vskip.1cm \noindent
Let now $k_t$, as a function of $\eta \in \Gamma_0$, be continuous. Then instead of (\ref{2D}) we
consider
\begin{equation*}
\widetilde{\mathcal{K}}_\vartheta = \{k \in C (\Gamma_0\rightarrow
\mathbb{R}): \|k\|_\vartheta < \infty\},
\end{equation*}
endowed with the same norm as in (\ref{o5}), (\ref{0o5}).
\begin{corollary}
  \label{o1rk}
Let  $\vartheta$, $T$, and $\vartheta_0$ be as in Theorem \ref{1tm}. Suppose
in addition that the function $\phi$  is continuous. Then  the
problem (\ref{R4})
 with $k_0\in \widetilde{\mathcal{K}}_{\vartheta_0}$
has a unique classical solution $k_t \in
\widetilde{\mathcal{K}}_{\vartheta}$ on $[0, T)$.
\end{corollary}
Now we consider (\ref{16A})
in the Banach space
\begin{equation*}
\mathcal{G}_\vartheta = L^1 (\Gamma_0, e^{-\vartheta |\cdot|}d
\lambda),\qquad \vartheta \in \mathbb{R},
\end{equation*}
that is, $G\in \mathcal{G}_\vartheta$ if $\|G\|_\vartheta< \infty$, where
\begin{equation}
 \label{6}
\|G\|_\vartheta :=  \int_{\Gamma_0}
\exp(-\vartheta |\eta|) \left\vert G(\eta) \right\vert \lambda (d
\eta).
\end{equation}
\begin{theorem}
  \label{2tm}
Let $\vartheta_0$ $\vartheta$,  $T>0$ be as in (\ref{beta}). Then the Cauchy problem
(\ref{16A}) with $G_0 \in \mathcal{G}_{\vartheta}$ has a unique
classical solution $G_t\in \mathcal{G}_{\vartheta_0}$ on
$[0,T)$.
\end{theorem}
\begin{proof}
As above, we obtain the solution of (\ref{16A}) as the limit of the sequence
$\{G_{t,n}\}_{n\in \mathbb{N}_0 }\subset \mathcal{G}_{\vartheta_0}$,
where $G^{(0)}_t = G_0$ and
\begin{equation}
  \label{8C}
G_{t,n} = G_0 + \sum_{m=1}^n \frac{1}{m!}t^m
\widehat{L}^m G_0.
\end{equation}
For the norm (\ref{6}), from (\ref{C5b}) similarly as above by (\ref{12A}) we get
\begin{equation*}
\|  \widehat{L}G \|_{\vartheta''} \leq  \frac{2\alpha}{e(\vartheta'' -
\vartheta')} \exp\left(c_\phi e^{-\vartheta''} \right)
\|G\|_{\vartheta'}.
\end{equation*}
This means that $\widehat{L}$ can be defined as a bounded linear
operator $\widehat{L}:\mathcal{G}_{\vartheta'} \rightarrow
\mathcal{G}_{\vartheta''}$ with norm
\begin{equation*}
\|\widehat{L}\|_{\vartheta' \vartheta''} \leq
\frac{2\alpha}{e(\vartheta'' - \vartheta')} \exp\left(c_\phi
e^{-\vartheta} \right).
\end{equation*}
Then we apply the latter estimate in (\ref{8C}) and obtain, for any $p,n\in\mathbb N$,
\begin{equation*}
\| G_{t,n} - G_{t,n+p}\|_\vartheta \leq
\sum_{m=n+1}^{n+p} \frac{(m/e)^m}{m!} \left( \frac{t}{T}\right)^m.
\end{equation*}
The latter estimate yields the proof, as in the case of Theorem \ref{1tm}.
\end{proof}
\begin{corollary}
  \label{0oco}
Let $k_0$, $k_t$, and $G_0$, $G_t$ be as in Theorem \ref{1tm} and Theorem \ref{2tm}, respectively.
Then, cf. (\ref{C5d}), the following holds
\begin{equation}
  \label{0o12}
 \langle \! \langle G_0 , k_t \rangle \! \rangle = \langle \! \langle G_t , k_0 \rangle \! \rangle.
\end{equation}
That is, the evolutions described by these Theorems are dual.
\end{corollary}
\begin{proof}
By  (\ref{dualL}) and by (\ref{8CD}) and (\ref{8C}), we see that, for all $n\in \mathbb{N}$,
\[
 \langle \! \langle G_0 , k_{t,n} \rangle \! \rangle = \langle \! \langle G_{t,n} , k_0 \rangle \! \rangle.
\]
Then (\ref{0o12}) is obtained from the latter by passing to the limit $n\to +\infty$, since we have the norm-convergences
of the sequences $\{k_{t,n}\}$ and $\{G_{t,n}\}$.
\end{proof}

\subsection{The evolution of $\Gamma_0$-states}

We recall that the set of finite configurations $ \Gamma_0$, cf. (\ref{2U}), is a measurable subset of $\Gamma$.
By a $\Gamma_0$-state we mean  a state $\mu \in \mathcal{M}^1(\Gamma)$ such that $\mu (\Gamma_0) =1$.
That is, in a $\Gamma_0$-state the system consists of a finite number of particles, but this number is random.
Each $\Gamma_0$-state  can be redefined as a probability measure on $(\Gamma_0, \mathcal{B}(\Gamma_0))$, cf.
(\ref{K2}) and (\ref{0K2}). The action  of the `generator' in (\ref{R1}) on $\Gamma_0$-states can be written down explicitly.
Namely, for such a state $\mu$ and $A \in \mathcal{B}(\Gamma_0)$,
\begin{equation}
  \label{0K3}
 (L^* \mu)(A) = - \int_{\Gamma_0} \Omega(\Gamma_0, \eta) \mathbb{I}_A (\eta) \mu(d\eta) +
 \int_{\Gamma_0} \Omega(A, \eta)  \mu(d\eta),
\end{equation}
where, cf. (\ref{C3}) and (\ref{C5}),
\begin{equation}
  \label{0K4}
\Omega (A, \eta) = \sum_{x\in \eta} \int_{\mathbb{R}^d} a(x-y) \exp( - E^\phi (y, \eta))\mathbb{I}_{A}(\eta \setminus x \cup y) dy,
\end{equation}
which is a measure kernel on $(\Gamma_0,\mathcal{B}(\Gamma_0))$. That is, $\Omega (\cdot, \eta)$ is a  measure for all $\eta \in \Gamma_0$, and
$\Omega (A, \cdot)$ is $\mathcal{B}(\Gamma_0)$-measurable for all $A \in \mathcal{B}(\Gamma_0)$.
Note that
\begin{equation}
  \label{0K5}
\Omega (\Gamma_0, \eta) =  \sum_{x\in \eta} \int_{\mathbb{R}^d} a(x-y) \exp( - E^\phi (y, \eta)) dy \leq \alpha |\eta|,
\end{equation}
which is obtained by (\ref{C4}) and the positivity of $\phi$.

Let $\mathcal{M}(\Gamma_0)$ be the Banach space of all signed measures on
$(\Gamma_0,\mathcal{B}(\Gamma_0))$ which have bounded variation.
For each $\mu\in\mathcal{M}(\Gamma_0)$, there exist $\beta_{\pm} \geq 0$ and probability measures $\mu_{\pm}$ such that
\begin{equation}
  \label{y01}
\mu= \beta_{+} \mu_{+} - \beta_{-} \mu_{-}, \qquad {\rm and } \qquad  \| \mu \| = \beta_{+} + \beta_{-}.
\end{equation}
Let $\mathcal{M}_{+} (\Gamma_0)$ be the cone of positive elements of $\mathcal{M}(\Gamma_0)$, for which
$\|\mu \| = \mu(\Gamma_0)$.
Then we define, cf. (\ref{0K5}),
\begin{equation}
  \label{0K6}
{\rm Dom} (L^*) = \{ \mu \in \mathcal{M}(\Gamma_0): \Omega (\Gamma_0, \cdot )\mu \in \mathcal{M}(\Gamma_0)\}.
\end{equation}
Recall that a $C_0$-semigroup  $\{S_\mu (t)\}_{t\geq 0}$ of bounded operators
in $\mathcal{M}(\Gamma_0)$ is called {\it stochastic} if each $S_\mu (t)$, $t>0$, leaves the cone $\mathcal{M}_{+}(\Gamma_0)$ invariant, and
$\|S_\mu (t)\mu\| =1$ whenever $\|\mu\|=1$. Our aim is to show that the problem (\ref{R1}) has a solution in the form
\begin{equation}
 \label{yyz}
\mu_t = S_\mu (t) \mu_0,
\end{equation}
where $\{S_\mu (t)\}_{t\geq 0}$ is a stochastic semigroup in $\mathcal{M}(\Gamma_0)$, that leaves invariant important subspaces of $\mathcal{M}(\Gamma_0)$.
For a measurable $b: \Gamma_0 \to \mathbb{R}_{+}$, we set
\begin{equation*}
 \mathcal{M}_b (\Gamma_0) = \{ \mu \in \mathcal{M}(\Gamma_0):  \mu_{\pm} (b) < \infty \},
\end{equation*}
where $\mu_{\pm}$ are the same as in (\ref{y01}) and
\begin{equation*}
  \mu_{\pm}(b) := \int_{\Gamma_0} b(\eta) \mu_{\pm}(d\eta).
\end{equation*}
The set $\mathcal{M}_b (\Gamma_0)$ can be equipped with the norm
\begin{equation*}
 \|\mu\|_b = \alpha_{+} \mu_{+} (b) + \alpha_{-} \mu_{-} (b),
\end{equation*}
which turns it into a Banach space. Set
\begin{equation*}
\mathcal{M}_{b,+} (\Gamma_0) = \mathcal{M}_{b} (\Gamma_0) \cap \mathcal{M}_{+} (\Gamma_0).
\end{equation*}
We also suppose that $b$ is such that the embedding $ \mathcal{M}_b (\Gamma_0) \hookrightarrow  \mathcal{M} (\Gamma_0)$
is dense and continuous. In the sequel, we use  \cite[Proposition 5.1]{[TV]}, which we rephrase as follows.
\begin{proposition}
  \label{TV1pn}
Suppose that $b$ and some positive $C$ and $\varepsilon$ obey the following estimate
\begin{equation}
\label{yy7}
\int_{\Gamma_0} \left(b(\xi) - b(\eta)\right)\Omega(d \xi, \eta) \leq C b(\eta) - \varepsilon \Omega (\Gamma_0, \eta),
\end{equation}
which holds for all $\eta \in \Gamma_0$. Then the closure of $L^*$ as in (\ref{0K3}), (\ref{0K4}) with domain (\ref{0K6}) generates a $C_0$-stochastic semigroup $\{S_\mu (t)\}_{t\geq 0}$,
which leaves $\mathcal{M}_{b} (\Gamma_0)$ invariant and induces a positive $C_0$-semigroup on $(\mathcal{M}_{b} (\Gamma_0), \|\cdot\|_b)$.
\end{proposition}
\begin{theorem}
  \label{0Ktm}
The problem (\ref{R1}) with a $\Gamma_0$-state $\mu_0$ has a unique classical solution in $\mathcal{M}_{+} (\Gamma_0)$  on $[0, +\infty)$, given by (\ref{yyz}) where $S_\mu (t)$, $t>0$, constitute the stochastic semigroup on $\mathcal{M} (\Gamma_0)$ generated by the closure of $L^*$ given in (\ref{0K3}), (\ref{0K4}), and (\ref{0K6}).
Moreover, for each $b$ which satisfies
\begin{equation}
  \label{0K9}
 b(\eta) = \delta (|\eta|) \geq \varepsilon \Omega(\Gamma_0, \eta), \qquad  {\rm for} \ \ {\rm all} \ \ \eta \in \Gamma_0,
\end{equation}
with some $\varepsilon>0$ and suitable $\delta:\mathbb{N} \to [0, +\infty)$, the mentioned semigroup $\{S_\mu (t)\}_{t\geq 0}$ leaves $\mathcal{M}_{b,+} (\Gamma_0)$ invariant.
\end{theorem}
\begin{proof}
Computations based on (\ref{12A}) show that, for $b(\eta) = \delta (|\eta|)$, the left-hand side of
(\ref{yy7}) vanishes, which  reflects the fact the the Kawasaki dynamics is conservative. Then the proof follows by Proposition \ref{TV1pn}. The condition that $\mu_0 \in \mathcal{M}_{b,+} (\Gamma_0)$ with $b$ satisfying (\ref{0K9}) merely means that this $\mu_0$ is taken from the domain of $L^*$, cf. (\ref{0K6}).
\end{proof}
Suppose now that the initial state $\mu_0$ in (\ref{R1}) is supported on $\Gamma_0$ and is absolutely continuous with respect to the Lebesgue-Poisson measure $\lambda$.
Then
\begin{equation}
  \label{0K10}
 R_0 (\eta) = \frac{d \mu_0}{d \lambda} (\eta)
\end{equation}
is a positive element of unit norm of the Banach space
$\mathcal{R} :=L^1 (\Gamma_0, d \lambda)$. If $\mu_0 \in \mathcal{M}_{b,+} (\Gamma_0)$, then also
$R_0 \in \mathcal{R}_b := L^1 (\Gamma_0,b d \lambda)$. For $b$ obeying (\ref{0K9}), it is possible to show that, for any $t>0$, the solution $\mu_t$
as in Theorem \ref{0Ktm} has the Radon-Nikodym derivative $R_t$ which lies in $\mathcal{R}_b$. Furthermore, there exists a stochastic semigroup
$\{S_R(t)\}_{t\geq 0}$ on $\mathcal{R}$, which leaves invariant each $\mathcal{R}_b$ with $b$ obeying (\ref{0K9}), and such that
\begin{equation}
  \label{0K11}
  R_{t} = S_R (t) R_0, \qquad t \geq 0.
\end{equation}
The generator $L^\dagger$ of the semigroup $\{S_R(t)\}_{t\geq 0}$ has the following properties
\begin{eqnarray}
  \label{0K11a}
 & & \qquad \qquad \qquad {\rm Dom} (L^\dagger) \supset \mathcal{R}_b,\\[.3cm]
 \label{0K11b}
& & \int_{\Gamma_0} F (\eta) (L^\dagger R)(\eta) \lambda (d \eta) = \int_{\Gamma_0}(L F )(\eta) R(\eta) \lambda (d \eta),  
\end{eqnarray}
which holds for each $b$ obeying (\ref{0K9}), and for each
$R\in \mathcal{R}$ and each measurable $F:\Gamma_0 \to \mathbb{R}$ such that both integrals in (\ref{0K11b}) exist. Here $L$ is as in (\ref{C2}).
For each $t \geq 0$, the correlation function of $\mu_t$ and its Radon-Nikodym derivative satisfy, cf. (\ref{9AA}),
\begin{equation}
  \label{0K12}
k_{\mu_t} (\eta) = \int_{\Gamma_0} R_t (\eta \cup \xi) \lambda (d\xi).
\end{equation}
By this representation and by (\ref{12A}), we derive
\[
\int_{\mathbb{R}^d} k^{(1)}_{\mu_t}(x) dx = \int_{\Gamma_0}|\eta| R_t (\eta ) \lambda (d\eta),
\]
which yields the expected number of particles in state $\mu_t$. Note that we cannot expect now that $k_{\mu_t}$ lies in the spaces where we solve
(\ref{R4}), cf. Theorem \ref{1tm}.

\subsection{The evolution of states}

Recall that by $B_{\rm bs}(\Gamma_0)$ we denote the set of all bounded measurable functions $G:\Gamma_0 \to \mathbb{R}$
each of which is supported on a bounded $A$, cf. (\ref{6A}). Its subset  $B^+_{\rm bs}(\Gamma_0)$ is defined in (\ref{bbs}).

Given $\vartheta\in \mathbb{R}$, let $\mathcal{M}_\vartheta(\Gamma)$ stand for
the set of all $\mu\in \mathcal{M}^1_{\rm fm}(\Gamma)$, for which $k_\mu \in \mathcal{K}_\vartheta$,
see (\ref{9AA}) and (\ref{2D}). Let also $\mathcal{K}_\vartheta^+$ be
the set of all $k \in \mathcal{K}_\vartheta$ such that, cf. (\ref{10A}),
\begin{equation}
\label{Q1}
 \int_{\Gamma_0} G(\eta) k(\eta) \lambda (d \eta) \geq 0,
\end{equation}
which holds for every $G\in B^+_{\rm bs}(\Gamma_0)$.
Note that this property is `more than the mere positivity' as $B^+_{\rm bs}(\Gamma_0)$ can contain functions which take also negative values, see
(\ref{7A}) and (\ref{bbs}).
Then
in view of Proposition \ref{Tobpn}, the map $\mathcal{M}_\vartheta (\Gamma)\ni \mu \mapsto
k_\mu \in \mathcal{K}_\vartheta^+$ is a bijection as such $k_\mu$ certainly obeys
(\ref{11A}). In what follows, the evolution of states $\mu_0 \mapsto \mu_t$ is understood as the evolution of the corresponding correlation functions
$k_{\mu_0} \mapsto k_{\mu_t}$ obtained by solving the problem (\ref{R4}).
\begin{theorem}
 \label{Qtm}
Let $\vartheta_0\in\mathbb R$, $\vartheta$, and $T(\vartheta)$ be as in Theorem \ref{1tm} and in (\ref{8}), respectively, and let $\mu_0$ be in $\mathcal{M}_{\vartheta_0}(\Gamma)$.
Then the evolution described in Theorem \ref{1tm} with
 $k_0 =k_{\mu_0}$ leaves $\mathcal{K}_{\vartheta}^+$ invariant, which means that each $k_t$ is the correlation function of
 a unique $\mu_t \in \mathcal{M}_\vartheta(\Gamma)$. Hence, the evolution $k_{\mu_0}\mapsto k_t$, $t\in [0,T(\vartheta))$, determines the
 evolution of states
\begin{displaymath}
 \mathcal{M}_{\vartheta_0}(\Gamma) \ni
\mu_0 \mapsto \mu_t \in \mathcal{M}_\vartheta(\Gamma), \quad t\in [0, T(\vartheta)).
\end{displaymath}
\end{theorem}
\begin{proof}
To prove the statement we have to show that a solution $k_t$ of the problem (\ref{R4}) with $k_0 = k_{\mu_0}$ obeys
(\ref{Q1}) for all $t\in (0,T(\vartheta))$. Fix $\mu_0 \in  \mathcal{M}_{\vartheta_0}(\Gamma)$ and take
$\Lambda \in \mathcal{B}_{\rm b}(\mathbb{R}^d)$.
Let $\mu^\Lambda_0$ be the projection of $\mu_0$ onto $\Gamma_{\Lambda}$, cf. (\ref{Lambda}).
Since $\mu_0$ is  in $\mathcal{M}^1_{\rm fm}(\Gamma)$, its density $R_0^\Lambda$, as in (\ref{0K10}), is in $\mathcal{R}$.
Given $N\in \mathbb{N}$, we let $I_N(\eta) = 1$ whenever $|\eta|\leq N$, and $I_N(\eta) = 0$ otherwise.  Then we set
\begin{equation}
  \label{0Q1}
  R^{\Lambda, N}_0 (\eta) = R^\Lambda_0 (\eta) I_N(\eta).
\end{equation}
As a function on $\Gamma_0$, $R^{\Lambda, N}_0$ is a collection of $R^{\Lambda, N,n}_0: (\mathbb{R}^d)^n \to \mathbb{R}_+$, $n\in \mathbb{N}_0$.
Clearly, $R^{\Lambda, N}_0$ is a positive element of $\mathcal{R}$ of norm $\|R^{\Lambda, N}_0\|_{\mathcal{R}} \leq 1$. Furthermore,
for each $\beta >0$,
\[
\int_{\Gamma_0} e^{\beta |\eta|} R^{\Lambda, N}_0 (\eta) \lambda (d\eta) = \sum_{n=0}^N \frac{e^{n \beta}}{n!}\int_{\Lambda^n} R_0^{\Lambda,N,n} (x_1 , \dots , x_n)d x_1 \cdots d x_n <\infty,
\]
and hence $R^{\Lambda, N}_0 \in \mathcal{R}_b$ with $b(\eta) = e^{\beta |\eta|}$, for each $\beta>0$.
Set
\begin{equation}
  \label{0Q2}
  R^{\Lambda, N}_t = S_R(t) R^{\Lambda, N}_0. \qquad t\geq 0,
\end{equation}
where $S_R(t)$ is as in (\ref{0K11}). By Theorem \ref{0Ktm}, we have that
\begin{eqnarray}
  \label{Rlambda}
\forall t\geq 0: \ \ & (a)& \forall \beta >0 \quad R^{\Lambda, N}_t \in \mathcal{R}_b  , \quad {\rm with} \ \ b(\eta) = e^{\beta |\eta|}, \\[.2cm]
& (b) & R^{\Lambda, N} (\eta ) \geq 0, \ \  {\rm for} \ \lambda-{\rm a.a.} \ \eta\in \Gamma_0, \nonumber \\[.2cm]
& (c) & \|R^{\Lambda, N}_t\|_{\mathcal{R}} \leq 1.    \nonumber
\end{eqnarray}
Furthermore, in view of (\ref{0K11a}), by \cite[Theorem 2.4, pp. 4--5]{Pasy} we have from (\ref{0Q2})
\begin{equation}
  \label{0Q3}
  R^{\Lambda, N}_t = R^{\Lambda, N}_0 + \int_0^t L^\dagger R^{\Lambda, N}_s ds.
\end{equation}
Set, cf. (\ref{9AA}) and (\ref{0K12}),
\begin{equation}
 \label{Q2}
q_{t}^{\Lambda,N} (\eta) = \int_{\Gamma_0} R^{\Lambda, N}_t (\eta \cup \xi) \lambda(d \xi).
\end{equation}
For $G\in B_{\rm bs}(\Gamma_0)$, let $N(G)\in \mathbb{N}_0$ be such that $G^{(n)} \equiv 0$ for $n> N(G)$. For such $G$, $KG$ is a cylinder
function on $\Gamma$, which can also be considered as a measurable function on $\Gamma_0$. By (\ref{7A}), we have that, for every $G \in B_{\rm bs}(\Gamma_0)$ and each $t\geq 0$,
\begin{equation}
  \label{0Q4}
\langle \! \langle K G, R^{\Lambda, N}_t  \rangle \! \rangle = \langle \! \langle G, q^{\Lambda, N}_t  \rangle \! \rangle,
\end{equation}
see (\ref{19A}).  Since $G\in B_{\rm bs}(\Gamma_0)$ is bounded, we have
\begin{equation}
  \label{0Q5}
 C(G) := \max_{n\in \{0, \dots , N(G)\}}\| G^{(n)}\|_{L^\infty((\mathbb{R}^d)^n)}<\infty,
\end{equation}
which immediately yields that
\begin{equation*}
  |(KG) (\eta)| \leq (1+|\eta|)^{N(G)} C(G),
\end{equation*}
and hence both integrals in (\ref{0Q4}) exist since $R^{\Lambda,N}_t\in \mathcal{R}_b$  for $b(\eta) = e^{\beta|\eta|}$
with any $\beta>0$. Moreover, by the same argument the map $\mathcal{R} \ni R \mapsto \langle \! \langle
 KG , R \rangle \! \rangle$ is continuous, and thus from (\ref{0Q3}) and  (\ref{0K11b}) we obtain
 \begin{eqnarray}
   \label{y12}
\langle \! \langle KG ,R^{\Lambda, N}_t \rangle \! \rangle & = &   \langle \! \langle KG ,R^{\Lambda,N}_0 \rangle \! \rangle
+  \int_0^t    \langle \! \langle KG ,{L}^\dagger R^{\Lambda,N}_s \rangle \! \rangle ds \\[.2cm]
& = &   \langle \! \langle KG ,R^{\Lambda,N}_0 \rangle \! \rangle +  \int_0^t
\langle \! \langle L KG,R^{\Lambda,N}_s \rangle \! \rangle ds. \nonumber
\end{eqnarray}
Now we would want to interchange in the latter line $L$ and $K$. If $\widehat{L} G$ were in $B_{\rm bs}(\Gamma_0)$,
one could get point-wise $LKG = K \widehat{L} G$ -- by the very definition of $\widehat{L}$. However, this is not the case since, cf. (\ref{C5b}),
\[
|(\widehat{L} G)(\eta)| \leq (KUG) (\eta),
\]
where
\[
(UG)(\xi) = \sum_{x\in \xi} \int_{\mathbb{R}^d}a(x-y) \left\vert G(\xi \setminus x \cup y) - G(\xi)\right\vert dy.
\]
Here we used,  cf. (\ref{4A}) and (\ref{18A}), that
\[
0\leq e(\tau_y,\xi) \leq 1, \qquad |e(t_y, \eta \setminus \xi)| \leq 1,
\]
which holds for almost all all $y$, $\xi$, and $\eta$.
Then, for $G\in B_{\rm bs}(\Gamma_0)$, we have, cf (\ref{0Q5}),
\[
N(UG) = N(G), \qquad  \ C(UG)\leq  2 \alpha N(G) C(G),
\]
which then yields
\begin{equation}
  \label{0Q50}
 | (\widehat{L} G)(\eta)| \leq 2\alpha N(G) C(G) (1+ |\eta|)^{N(G)}.
\end{equation}
Let us show that, for any $t\geq 0$, the function $(\widehat{L} G) q^{\Lambda,N}_t$ is $\lambda$-integrable, cf. (\ref{Q2}). By (\ref{12A}), from
(\ref{Q2}) and (\ref{0Q50}) we get
\begin{eqnarray*}
\langle \! \langle \widehat{L} G, q^{\Lambda, N}_t  \rangle \! \rangle & \leq & 2 \alpha   N(G) C(G) \int_{\Gamma_0}  R^{\Lambda,N}_t(\eta)
\left(\sum_{\xi\subset \eta} (1+ |\xi|)^{N(G)}\right) \lambda(d \eta) \qquad \\[.2cm]
& \leq & 2 \alpha   N(G) C(G) \int_{\Gamma_0} 2^{|\eta|} (1+ |\eta|)^{N(G)}R^{\Lambda,N}_t(\eta) \lambda(d \eta). \nonumber
\end{eqnarray*}
 Hence, by claim $(a)$ in (\ref{Rlambda}) we get the integrability in question. Then
by (\ref{0Q4}) we transform (\ref{y12}) into
\begin{equation}
  \label{0Q6}
  \langle \! \langle G, q^{\Lambda, N}_t  \rangle \! \rangle = \langle \! \langle G, q^{\Lambda, N}_0  \rangle \! \rangle
  + \int_0^t \langle \! \langle \widehat{L} G, q^{\Lambda, N}_s  \rangle \! \rangle ds.
\end{equation}
Since $R^{\Lambda,N}_t$ is positive, cf. $(b)$ in (\ref{Rlambda}), by (\ref{0Q4})
 we get
\begin{equation}
  \label{0Q7}
 \langle \! \langle G, q^{\Lambda, N}_t  \rangle \! \rangle \geq 0 \quad \qquad {\rm for} \ \  G\in B^+_{\rm bs}(\Gamma_0).
\end{equation}
On the other hand, by (\ref{0Q1}) and (\ref{Q2}) we have, see also (\ref{9AA}),
\begin{equation}
  \label{0Q8}
  0 \leq q_0^{\Lambda,N}(\eta) \leq \int_{\Gamma_0} R^\Lambda (\eta \cup \xi) \lambda (d\xi) = k_{\mu_0}(\eta) \mathbb{I}_{\Gamma_\Lambda}(\eta) \leq k_{\mu_0}(\eta),
\end{equation}
where $\mathbb{I}_{\Gamma_\Lambda}$ is the indicator of $\Gamma_\Lambda$, i.e., $\mathbb{I}_{\Gamma_\Lambda}(\eta) = 1$ whenever $\eta \in \Gamma_\Lambda$, and $\mathbb{I}_{\Gamma_\Lambda}(\eta)=0$ otherwise. By (\ref{0Q8}),
$q_0^{\Lambda,N} \in \mathcal{K}_{\vartheta_0}$. Let $k^{\Lambda,N}_t$, $t\in [0,T)$, be the solution of (\ref{R4}) with $k_0 = q_0^{\Lambda,N}$, as stated in Theorem \ref{1tm}. Then
\[
k^{\Lambda,N}_t = k^{\Lambda,N}_0 + \int_0^t L^\Delta k^{\Lambda,N}_s ds,
\]
which for $G$ as in (\ref{0Q6}) yields
\begin{equation}
  \label{0Q9}
 \langle \! \langle G, k^{\Lambda, N}_t  \rangle \! \rangle = \langle \! \langle G, q^{\Lambda, N}_0  \rangle \! \rangle
  + \int_0^t \langle \! \langle \widehat{L} G, k^{\Lambda, N}_s  \rangle \! \rangle ds.
\end{equation}
Set
\begin{equation*}
\varphi(t; G) =  \langle \! \langle G, q^{\Lambda, N}_t  \rangle \! \rangle, \qquad \psi(t;G) =  \langle \! \langle G, k^{\Lambda, N}_t  \rangle \! \rangle.
\end{equation*}
By (\ref{0Q6}) and (\ref{0Q9}), we obtain, cf. Corollary \ref{0oco},
\begin{equation}
  \label{0Q11}
  \frac{d^n\varphi}{dt^n } (0; G) = \frac{d^n \psi}{dt^n } (0; G) = \langle \! \langle \widehat{L}^n G, q^{\Lambda, N}_0  \rangle \! \rangle =
  \langle \! \langle  G,( L^\Delta)^n q^{\Lambda, N}_0  \rangle \! \rangle.
\end{equation}
From this we can get that, cf. (\ref{0Q7}),
\begin{equation}
  \label{0Q12}
  \langle \! \langle G, k^{\Lambda, N}_t  \rangle \! \rangle = \langle \! \langle G, q^{\Lambda, N}_t  \rangle \! \rangle \geq 0, \qquad {\rm for} \ \  G\in B^+_{\rm bs}(\Gamma_0),
\end{equation}
provided the series
\[
\sum_{m=0}^\infty \frac{t^m}{m!} \langle \! \langle  G,( L^\Delta)^m q^{\Lambda, N}_0  \rangle \! \rangle
\]
converges for all $t\in [0,T(\vartheta))$, cf. (\ref{0Q11}). But the latter indeed holds true in view of (\ref{Lest}), which implies that (\ref{0Q12}) holds for all
$t\in [0, T(\vartheta))$.

In Appendix, we show that, for each $G\in B_{\rm bs}^+ (\Gamma_0)$ and any $t\in [0, T(\vartheta))$,
\begin{equation}
  \label{0Q13}
    \langle \! \langle G, k_t  \rangle \! \rangle = \lim_{n\to +\infty} \lim_{l\to +\infty}   \langle \! \langle G, k^{\Lambda_n, N_l}_t  \rangle \! \rangle,
\end{equation}
for certain increasing sequences $\{\Lambda_n\}_{n\in \mathbb{N}}$ and $\{N_l\}_{l\in \mathbb{N}}$ such that $N_l \to + \infty$ and $\Lambda_n \to \mathbb{R}^d$. Then by (\ref{0Q13}) and (\ref{0Q12}) we obtain (\ref{Q1}), and thus complete the proof.
\end{proof}

\section{Mesoscopic dynamics}
\label{Sec4}

As mentioned above, the mesoscopic description of the considered
model is obtained by means of a Vlasov-type scaling, originally developed for describing mesoscopic
properties of plasma. We refer to \cite{Dob,P,spohn} as to
the source of general concepts in this field, as well as to  \cite{DimaN} where the peculiarities of the scaling method which we use are given along
with the updated bibliography on this item.

\subsection{The  Vlasov hierarchy}

The main idea of the scaling which we use in this article is to make the particle system
more and more dense whereas the interaction respectively
weaker. This corresponds to the so called mean field approximation
widely employed in theoretical physics.
Note that we are not scaling time, which would be the case for a macroscopic scaling.
The object of our
manipulations will be the problem (\ref{R4}). The scaling parameter
$\varepsilon>0$  will be tending to zero. The first
step is to assume that the initial state depends on $\varepsilon$ in
such a way that the correlation function $k_0^{(\varepsilon)}$ diverges as $\varepsilon
\rightarrow 0$ in such a way that
the so called renormalized correlation function
\begin{equation}
  \label{C9a}
k_{0, {\rm ren}}^{(\varepsilon)} (\eta) := \varepsilon^{|\eta|} k_0^{(\varepsilon)}
\end{equation}
converges $k_{0, {\rm ren}}^{(\varepsilon)}\rightarrow r_0$, as $\varepsilon \rightarrow 0$, to the correlation function
of a certain measure.
Let  $k_0^{(\varepsilon,n)}:(\mathbb{R}^d)^n \to \mathbb{R}$ denote  $n$-particle `component' of $k_0^{(\varepsilon)}$. Then our assumption,
in particular, means
\begin{equation}
  \label{VSa}
  k_0^{(\varepsilon,1)} \sim \varepsilon^{-1}.
\end{equation}
Then the second step is to consider the Cauchy problem
\begin{equation}
  \label{C10}
\frac{d}{dt}k^{(\varepsilon)}_t = L_{\varepsilon}^\Delta k^{(\varepsilon)}_t, \qquad
k^{(\varepsilon)}_t|_{t=0} = k^{(\varepsilon)}_0,
\end{equation}
where $L_{\varepsilon}^\Delta$ is as in (\ref{C5c}) but with $\phi$ multiplied by $\varepsilon$.
As might be seen from (\ref{8CD}), the solution $k^{(\varepsilon)}_t$, which exists in view of Theorem \ref{1tm},
diverges as $\varepsilon \rightarrow 0$. Thus, similarly as in (\ref{C9a}) we pass to
\begin{equation}
  \label{C11}
k_{t, {\rm ren}}^{(\varepsilon)} (\eta) = \varepsilon^{|\eta|} k_t^{(\varepsilon)},
\end{equation}
which means that instead of (\ref{C10}) we are going to solve the following problem
\begin{equation}
  \label{C12}
\frac{d}{dt}k^{(\varepsilon)}_{t, {\rm ren}} = L_{\varepsilon,{\rm ren}} k^{(\varepsilon)}_{t,{\rm ren}} \qquad
k^{(\varepsilon)}_{t,{\rm ren}}|_{t=0} = k^{(\varepsilon)}_{0,{\rm ren}},
\end{equation}
with
\begin{equation}
  \label{C13}
 L_{\varepsilon,{\rm ren}} = R_\varepsilon^{-1} L_{\varepsilon}^\Delta R_\varepsilon,\qquad
 \left(R_\varepsilon k\right) (\eta) := \varepsilon^{-|\eta|} k(\eta).
\end{equation}
\begin{remark}
  \label{SP1rk}
Since $k^{(\varepsilon)}_0$ is a correlation function, by Theorem \ref{0Ktm} we know that $k^{(\varepsilon)}_t$ is the correlation function of a unique measure
$\mu^{(\varepsilon)}_t$. If $\mu_0^{(1)}$ is a Poisson measure with density
$k_0^{(1,1)} = \varrho_0$, then also $\mu_0^{(\varepsilon)}$ with density $k_0^{(\varepsilon,1)} = \varepsilon^{-1}\varrho_0$ is a Poisson measure.
We can expect that, for $t>0$, $k^{(\varepsilon)}_{t, {\rm ren}}$ has a nontrivial limit as $\varepsilon \to 0^+$, only if
$k^{(\varepsilon)}_{t }(\eta) \leq [k^{(\varepsilon,1)}_{t }(x)]^{|\eta|}$, cf. (\ref{VSa}) and (\ref{C11}). For this to hold,  $\mu^{(\varepsilon)}_t$ should be sub-Poissonian, cf. Definition \ref{SPdf} and Remark \ref{SPrk}. That is, the evolution $\mu_0^{(\varepsilon)} \mapsto \mu_t^{(\varepsilon)}$ should preserve sub-Poissonicity, which is the case by Theorem \ref{1tm} in view of (\ref{o5}).
\end{remark}
By (\ref{C5c}) and (\ref{C13}), we have
\begin{eqnarray}
  \label{C14}
\left( {L}_{\varepsilon,{\rm ren}} k\right) (\eta) & = & \sum_{y\in \eta} \int_{\mathbb{R}^d} a (x-y)
 e( \tau^{(\varepsilon)}_y, \eta \setminus y \cup x)\\[.2cm] & \times & \left( \int_{\Gamma_0}e(\varepsilon^{-1}t^{(\varepsilon)}_y, \xi)
  k( \xi \cup x \cup \eta \setminus y)
 \lambda (d \xi) \right) dx \nonumber \\[.2cm] & - & \int_{\Gamma_0} k(\xi \cup \eta)
  \left( \sum_{x\in \eta} \int_{\mathbb{R}^d} a(x-y) e(\tau^{(\varepsilon)}_y,\eta) \right.\nonumber\\[.2cm]
  & \times &   e(\varepsilon^{-1} t^{(\varepsilon)}_y, \xi) d y \bigg{)} \lambda (d\xi),  \nonumber
\end{eqnarray}
where, cf. (\ref{18A}),
\begin{equation*}
    t^{(\varepsilon)}_x (y) = e^{-\varepsilon \phi (x-y)} - 1, \qquad \tau^{(\varepsilon)}_x (y) = t^{(\varepsilon)}_x (y) + 1.
\end{equation*}
As in (\ref{C9}), for any $\vartheta'\in \mathbb{R}$ and
$\vartheta'' > \vartheta'$, we have
\begin{equation}
  \label{C16}
  \|L_{\varepsilon,{\rm ren}}  \|_{\vartheta'' \vartheta'} \leq \frac{2\alpha}{e(\vartheta'' - \vartheta')}
\exp\left(c^{(\varepsilon)}_\phi e^{-\vartheta''} \right),
\end{equation}
where, cf. (\ref{13A}),
\begin{equation*}
 c^{(\varepsilon)}_\phi = \varepsilon^{-1}\int_{\mathbb{R}^d} \left( 1 - e^{-\varepsilon\phi(x)}\right)d x .
\end{equation*}
Suppose now that $\phi$ is in $L^1(\mathbb{R}^d)$ and set
\begin{equation*}
\langle \phi \rangle = \int_{\mathbb{R}^d} \phi(x) d x.
\end{equation*}
Recall that we still assume $\phi \geq 0$. Then
\begin{equation}
  \label{C19}
  \|L_{\varepsilon,{\rm ren}}  \|_{\vartheta'' \vartheta'} \leq \sup_{\varepsilon >0} \{  {\rm RHS(\ref{C16})}\} = \frac{2\alpha}{e(\vartheta'' - \vartheta')}
\exp\left(\langle \phi \rangle e^{-\vartheta''} \right).
\end{equation}
Let us now,
informally, pass in (\ref{C14}) to the limit $\varepsilon
\rightarrow 0$. Then we get the following operator
\begin{eqnarray}
  \label{C21}
\left( L_V k\right) (\eta) & = & \sum_{y\in \eta}
\int_{\mathbb{R}^d}
 a (x-y)\int_{\Gamma_0}e( - \phi(y-\cdot), \xi)\\[.2cm] & \times &
  k( \xi \cup x \cup \eta \setminus y)
 \lambda (d \xi) dx \nonumber \\[.2cm] & - & \int_{\Gamma_0} k(\xi \cup \eta)
  \sum_{x\in \eta} \int_{\mathbb{R}^d} a(x-y) \nonumber\\[.2cm]
  & \times &   e(- \phi(y-\cdot), \xi) d y  \lambda (d\xi).  \nonumber
\end{eqnarray}
It certainly obeys
\begin{equation}
  \label{C190}
  \|L_V  \|_{\vartheta'' \vartheta'} \leq  \frac{2\alpha}{e(\vartheta'' - \vartheta')}
\exp\left(\langle \phi \rangle e^{-\vartheta''} \right),
\end{equation}
and hence along with (\ref{C10}) we can consider the problem
\begin{equation}
  \label{C22}
 \frac{d}{dt} r_t = L_V r_t, \qquad r_t|_{t=0} = r_0,
\end{equation}
which is
 called the {\it Vlasov hierarchy} for the
Kawasaki system which we consider. Repeating the arguments used in the proof
of Theorem
\ref{1tm} we obtain the following
\begin{proposition}
  \label{Vlpn}
For  every $\vartheta_0\in \mathbb{R}$, there exists $T_* = T_* (\vartheta_0, \alpha, \langle \phi \rangle)$ such that
the problem (\ref{C12}) (resp. (\ref{C190})) with any $\varepsilon>0$ and $k^{(\varepsilon)}_0\in \mathcal{K}_{\vartheta_0}$ (resp.  $r_0 \in \mathcal{K}_{\vartheta_0}$) has a unique classical solution $k^{(\varepsilon)}_t\in \mathcal{K}_{\vartheta(t)}$ (resp. $r_t\in \mathcal{K}_{\vartheta(t)}$)
for $t\in [0, T_*)$.
\end{proposition}
As mentioned in Remark \ref{SP1rk}, $k^{(\varepsilon)}_t$ is also a correlation function if $k^{(\varepsilon)}_0$ is so.
However, this could not be the case
for $r_t$, even if $r_0 = k^{(\varepsilon)}_0$. Moreover, we do not
know how `close'  is $r_t$  to $k^{(\varepsilon)}_t$, as the passage from $L_{\varepsilon, {\rm ren}}$ to $L_V$ was only
informal. In the remaining part of the article we give answers to both these questions.

\subsection{The Vlasov equation}

Here we  show that the problem (\ref{C22}) has a very particular solution, which gives sense to the whole
construction.
For $a$ as in (\ref{C3}) and an appropriate $g:\mathbb{R}^d \rightarrow \mathbb{R}$, we write
\[
(a\ast g)(x) = \int_{\mathbb{R}^d} a(x-y) g(y) d y,
\]
and similarly for $\phi \ast g$. Then let us consider in $L^\infty (\mathbb{R}^d)$ the following Cauchy problem
\begin{eqnarray}
 \label{C24}
\frac{d}{dt}\varrho_t(x) & = & \left(a \ast \varrho_t \right)(x) \exp\left[- (\varrho_t \ast \phi)(x)\right]\\[.2cm] & - &
\varrho_t (x) \left(a \ast \exp\left( - \varrho_t \ast \phi\right) \right)(x), \nonumber \\[.2cm]
 \varrho_t|_{t=0} & = &  \varrho_0.  \nonumber
\end{eqnarray}
Denote
\begin{eqnarray*}
 \varDelta^+ & = & \{  \varrho \in L^\infty (\mathbb{R}^d): \varrho (x) \geq 0 \quad  {\rm for} \ \ {\rm a. a.} \ \ x\}, \\[.2cm]
 \varDelta_u & = & \{ \varrho \in L^\infty (\mathbb{R}^d) :   \|\varrho \|_{L^\infty(\mathbb{R}^d)} \leq u\}, \qquad u>0,\nonumber \\[.2cm]
 \varDelta^{+}_u & = & \varDelta^{+} \cap \varDelta_u. \nonumber
\end{eqnarray*}
\begin{lemma}
  \label{Vlm}
  Let $\vartheta_0$ and $T_*$ be as in Proposition \ref{Vlpn}.
Suppose that, for some $T\in (0,T_*)$, the problem (\ref{C24}) with $\varrho_0 \in \varDelta^{+}_{u_0}$,
has a unique classical solution $\varrho_t\in \varDelta^{+}_{u_T}$ on $[0,T]$, for some $u_T>0$.
Then, for $\vartheta_0 = - \log u_0$ and $\vartheta(T)= - \log u_T$,  the solution $r_t\in \mathcal{K}_{\vartheta(T)}$ of (\ref{C190}) as in Proposition \ref{Vlpn} with
$r_0 (\eta ) = e(\varrho_0 , \eta)$ is given by
\begin{equation}
 \label{C23}
r_t (\eta) = e(\varrho_t, \eta)= \prod_{x\in \eta} \varrho_t(x).
\end{equation}
\end{lemma}
\begin{proof}
First of all we note that, for a given $\vartheta$, $e(\varrho, \cdot) \in \mathcal{K}_\vartheta$ if and only if $\varrho \in \varDelta_u$ with $u= e^{-\vartheta}$,
see (\ref{o5}). Now set $\tilde{r}_t = e(\varrho_t, \cdot)$ with $\varrho_t$ solving (\ref{C24}).
This $\tilde{r}_t$ solves (\ref{C22}), which can easily be checked by computing $d/dt$ and
employing (\ref{C24}). In view of the uniqueness as in Proposition \ref{Vlpn}, we then have $\tilde{r}_t = r_t$ on $[0,T]$, from which it can be continued to $[0,T_*)$.\end{proof}
\begin{remark}
  \label{Vrk}
As
(\ref{C23}) is the correlation function of the Poisson measure $\pi_{\varrho_t}$, see (\ref{3Aa}) and (\ref{3Ab}), the above lemma establishes the so called {\it chaos preservation} or {\it chaos propagation} in time. Indeed, the most chaotic states those corresponding to Poisson measures, cf. (\ref{chi}), (\ref{11Aa}), and (\ref{11Ab}).
\end{remark}
Let us show now that  the problem (\ref{C24}) does have the solution we need. In a standard way, (\ref{C24})
can be transformed into the following integral equation
\begin{eqnarray}
 \label{C25}
\qquad \varrho_t(x) & = & F_t(\varrho)(x)
 :=
 \varrho_0(x)e^{-\alpha t}\\[.2cm] & + & \int_0^t \exp\left( - \alpha (t-s)\right)
 \left(a \ast \varrho_s \right)(x) \exp\left[- (\varrho_s \ast \phi)(x)\right]ds \nonumber \\[.2cm]
 & + & \int_0^t \exp\left( - \alpha (t-s)\right) \varrho_s (x)
\left[a \ast\left( 1- \exp\left( - \varrho_s \ast \phi\right) \right)\right](x) ds, \nonumber
\end{eqnarray}
that is, $[0,T)\ni t \mapsto \varrho_t \in L^\infty (\mathbb{R}^d)$ is a classical solution of (\ref{C24}) if and only if it solves
(\ref{C25}). Suppose $\varrho_t \in \varDelta^+$ is such a solution. Then we set
\begin{equation}
  \label{0C}
u_t = \|\varrho_t \|_{L^\infty(\mathbb{R}^d)}, \qquad t \in   [0,T).
\end{equation}
Since $\phi \geq 0$ and $\varrho_t \in \varDelta^+$, from (\ref{C25}) we get for $v_t:= u_t \exp(\alpha t)$, cf. (\ref{C4}),
 \[
v_t \leq v_0 + 2 \alpha \int_0^t v_s ds,
 \]
from which by the Gronwall inequality we obtain $v_t \leq v_0 \exp( 2 \alpha t)$; and hence,
\begin{equation}
  \label{0C1}
  u_t \leq u_0 e^{\alpha t}.
\end{equation}
In a similar way, one shows that, for $\varrho_0 \in \varDelta^{+}_{u_0}$ and $\varrho_s \in \varDelta^{+}_{u_t}$ for all $s\in [0,t]$,
\begin{equation}
  \label{0D}
 F_t (\varrho) \in \varDelta^{+}_{ u_t} , \ \  \qquad u_t:= \frac{u_0}{2- e^{\alpha t}}.
\end{equation}
Now for $\varrho_0\in \varDelta_{u_0}^{+}$ and some $t>0$ such that $e^{\alpha t} < 2$, cf. (\ref{0D}), we consider the sequence
 \begin{equation*}
 \varrho^{(0)}_t=\varrho_0,\qquad \varrho_t^{(n)}=F_t(\varrho^{(n-1)}), \quad n\in \mathbb{N}.
\end{equation*}
Obviously, each $\varrho^{(n)}_t$ is in $\varDelta^+_{u_t}$. Now let us find $T<\min\{T_*, \log 2/ \alpha\}$, $T_*$ being as in Lemma \ref{Vlm}, such  that the sequence of
\begin{equation}
  \label{0C2}
 \delta_n:=  \sup_{t\in [0,T]} \|\varrho_t^{(n)} - \varrho_t^{(n-1)}\|_{L^\infty (\mathbb{R}^d)}, \qquad n\in \mathbb{N}
\end{equation}
is summable, which would guarantee that, for each $t\leq T$, $\{\varrho_t^{(n)}\}_{n\in \mathbb{N}_0}$  is a Cauchy sequence. For $\varrho_s^{(n-1)}, \varrho_s^{(n-2)} \in \varDelta_{u_T}$, we have
\begin{eqnarray*}
 \bigg{\|} 1 - \exp\left( \phi \ast (\varrho_s^{(n-1)} - \varrho_s^{(n-2)})\right) \bigg{\|}_{L^\infty(\mathbb{R}^d)}  \leq  \bigg{\|} \phi \ast (\varrho_s^{(n-1)} - \varrho_s^{(n-2)}) \bigg{\|}_{L^\infty(\mathbb{R}^d)} & & \\[.2cm]
 \times  \sum_{m=0}^\infty \frac{1}{m!} \frac{1}{m+1}
\bigg{\|} \phi \ast (\varrho_s^{(n-1)} - \varrho_s^{(n-2)}) \bigg{\|}^m_{L^\infty(\mathbb{R}^d)} & & \qquad
\nonumber\\[.2cm]
 \leq \langle \phi \rangle \|\varrho_s^{(n-1)} - \varrho_s^{(n-2)} \|_{L^\infty (\mathbb{R}^d)}
 \exp \left(2 \langle \phi \rangle u_T   \right). \quad & &
\end{eqnarray*}
By means of this estimate, we obtain from (\ref{C25}) and (\ref{0C2})
\begin{equation*}
  \delta_n  \leq  q(T) \delta_{n-1},
\end{equation*}
where
\begin{equation}
 \label{Vc1}
q(T) = 2 ( 1 - e^{-\alpha T})\left( 1 + \langle \phi \rangle u_0 \exp\left(\alpha T + 2 \langle \phi \rangle u_0 e^{\alpha T}  \right)\right).
\end{equation}
Since $q(T)$ is a continuous increasing function such that $q(0) =0$, one finds $T>0$ such that $q(T) < 1$. For this $T$,
the sequence $\{\varrho_t^{(n)}\}_{n\in \mathbb{N}_0}$ converges to some $\varrho_t \in \varDelta_{u_T}^{+}$, uniformly
on $[0,T]$. Clearly, this $\varrho_t $ solves (\ref{C24}) and hence (\ref{C25}).
\begin{theorem}
  \label{Vatm}
The unique classical solution of (\ref{C22}) with $r_0 = e(\varrho_0, \cdot)$, $\varrho_0 \in \varDelta^{+}$, exists for all $t>0$ and
is given by (\ref{C23}) with $\varrho_t \in \varDelta^{+}$ being the solution of  (\ref{C24}).
Moreover, this solution
obeys
\begin{equation}
  \label{0V}
r_t  (\eta) \leq r_0  (\eta) \exp( \alpha |\eta| t).
\end{equation}
\end{theorem}
\begin{proof}
For a given $\varrho_0 \in \varDelta^+$, we find $T$ such that $q(T)<1$, cf. (\ref{0C}) and (\ref{Vc1}). Then there exists a unique classical solution of (\ref{C24}) $\varrho_t \in \varDelta^+_{u_T}$ on $[0,T]$, which by Lemma \ref{Vlm} yields the solution (\ref{C23}). Since $\varrho_t$ obeys the a priori bound (\ref{0C1}), it does not explode and hence can be continued, which yields also the continuation of $r_t$. Finally, the bound (\ref{0V}) follows from  (\ref{0C1}).
\end{proof}

\subsection{The scaling limit $\varepsilon \rightarrow 0$}

Our final task in this work is to show that the solution of (\ref{C12}) $k_{t,{\rm ren}}^{(\varepsilon)}$ converges in $\mathcal{K}_\vartheta$
uniformly on compact subsets of $[0,T_*)$ to that of (\ref{C22}), see Proposition \ref{Vlpn}. Here we should
impose an additional condition on the potential $\phi$, which, however, seems quite natural. Recall that in this section we
suppose $\phi\in L^1 (\mathbb{R}^d)$.
\begin{theorem}
  \label{3tm}
Let $\vartheta_0$ and $T_*$ be as in Proposition \ref{Vlpn}, and for $T\in [0,T_*)$,
take $\vartheta$ such that $T < T( \vartheta)$, see (\ref{8}).
Assume also that $\phi \in L^1 (\mathbb{R}^d) \cap L^\infty (\mathbb{R}^d)$
and consider the problems (\ref{C12}) and (\ref{C22}) with
$k^{(\varepsilon)}_{0, {\rm ren}} = r_0 \in
{\mathcal{K}}_{\vartheta_0}$. For their solutions $k^{(\varepsilon)}_{t, {\rm ren}}$ and  $r_t$,
it follows that  $k^{(\varepsilon)}_{t, {\rm
ren}} \rightarrow r_t$ in
${\mathcal{K}}_{\vartheta}$, as $\varepsilon \rightarrow 0$, uniformly on  $[0,T]$.
\end{theorem}
\begin{proof}
For $n\in \mathbb{N}$, let $k^{(\varepsilon)}_{t,n}$ and $r_{t,n}$ be defined as in (\ref{8CD}) with
$L_{\varepsilon, {\rm ren}}$ and $L_V$, respectively. As in the proof of Theorem \ref{1tm}, one can
show that the sequences of $k^{(\varepsilon)}_{t,n}$ and $r_{t,n}$ converge in $\mathcal{K}_\vartheta$
to $k^{(\varepsilon)}_{t,{\rm ren}}$ and $r_{t}$, respectively, uniformly on $[0,T]$. Then, for $\delta >0$,
one finds $n\in \mathbb{N}$ such that, for all $t\in[0,T]$,
\begin{equation*}
\|k^{(\varepsilon)}_{t,n} - k^{(\varepsilon)}_{t,{\rm ren}}\|_{\vartheta} + \| r_{t,n} - r_t\|_{\vartheta} < \delta/2.
\end{equation*}
From (\ref{8CD}) we then have
\begin{eqnarray}
  \label{U1}
\|k^{(\varepsilon)}_{t,{\rm ren}} - r_t\|_\vartheta & \leq & \bigg{\|}\sum_{m=1}^n \frac{1}{m!} t^m \left(L_{\varepsilon, {\rm ren}}^m - L^m_V \right)r_0\bigg{\|}_\vartheta + \frac{\delta}{2} \\[.2cm] & \leq &
\|L_{\varepsilon, {\rm ren}} - L_V\|_{\vartheta_0 \vartheta} \|r_0\|_{\vartheta_0} T \exp \left(T b(\vartheta)\right) + \frac{\delta}{2}, \nonumber
\end{eqnarray}
where, see (\ref{C19}) and (\ref{C190}),
\[
b(\vartheta) := \frac{2\alpha}{e(\vartheta_0 - \vartheta)} \exp\left(\langle \phi \rangle e^{-\vartheta}
 \right) .
\]
Here we used the following identity
\begin{eqnarray*}
& & L_{\varepsilon, {\rm ren}}^m - L^m_V  =  \left(L_{\varepsilon, {\rm ren}} - L_V \right)  L_{\varepsilon, {\rm ren}}^{m-1} + L_V \left(L_{\varepsilon, {\rm ren}} - L_V \right) L_{\varepsilon, {\rm ren}}^{m-2}\qquad \qquad  \\[.2cm] & & \qquad  +  \cdots + L_V^{m-2} \left(L_{\varepsilon, {\rm ren}} - L_V \right)  L_{\varepsilon, {\rm ren}} + L_V^{m-1} \left(L_{\varepsilon, {\rm ren}} - L_V \right) .\nonumber
\end{eqnarray*}
Thus, we have to show that
\begin{equation}\label{con}
 \|L_{\varepsilon, {\rm ren}}
-L_V\|_{\vartheta_0\vartheta}\to 0, \quad \ {\rm as} \ \  \varepsilon\to 0,
\end{equation}
which will allow us to make the first summand in the right-hand side of (\ref{U1}) also smaller than $\delta/2$ and thereby to complete the proof.

Subtracting (\ref{C21}) from (\ref{C14}) we get
\begin{eqnarray}
  \label{Vx}
\left(L_{\varepsilon, {\rm ren}}
-L_V \right)k(\eta) & = & \sum_{y \in \eta} \int_{\mathbb{R}^d} \int_{\Gamma_0} a(x-y) k(\xi \cup x \cup \eta \setminus y)\\[.2cm]
& \times & Q_{\varepsilon} (y,\eta \setminus y \cup x, \xi) \lambda (d \xi) d x \nonumber \\[.2cm] & - &  \sum_{x \in \eta} \int_{\mathbb{R}^d} \int_{\Gamma_0} a(x-y) k(\xi\cup \eta) \nonumber \\[.2cm] & \times & Q_{\varepsilon} (y,\eta ,\xi) \lambda (d \xi) d y \nonumber
\end{eqnarray}
where
\begin{eqnarray*}
Q_{\varepsilon} (y,\zeta ,\xi) & := & e(\tau_y^{(\varepsilon)}, \zeta)   e(\varepsilon^{-1}t^{(\varepsilon)}_y,\xi) - e (- \phi(y - \cdot), \xi)\\[.2cm]
 & = & e(\varepsilon^{-1}t^{(\varepsilon)}_y,\xi) - e (- \phi(y - \cdot) \nonumber \\[.2cm] & - & \left[  1 - e(\tau_y^{(\varepsilon)}, \zeta) \right] e(\varepsilon^{-1}t^{(\varepsilon)}_y,\xi). \nonumber
\end{eqnarray*}
For $t>0$, the function $e^{-t} - 1 + t$ takes positive values only; hence,
\[
\Psi(t) : = (e^{-t} - 1 + t)/t^2, \quad t>0,
\]
is positive and bounded, say by $C>0$.
Then by means of the inequality
\[
b_1 \cdots b_n - a_1 \cdots a_n \leq \sum_{i=1}^n (b_i - a_i) b_1 \cdots b_{i-1} b_{i+1} \cdots b_n, \quad b_i \geq a_1 >0,
\]
we obtain
\begin{eqnarray*}
 \left\vert e(\varepsilon^{-1}t^{(\varepsilon)}_y,\xi) - e (- \phi(y - \cdot), \xi) \right\vert & \leq& \sum_{z\in \xi} \varepsilon [\phi(y-z)]^2
\Psi\left(\varepsilon \phi(y-z)\right)\\[.2cm] & \times & \prod_{u\in \xi\setminus z} \phi(y-u) \\[.2cm] & \leq& \varepsilon C \sum_{z\in \xi}
 [\phi(y-z)]^2  e(\phi(y-\cdot), \xi\setminus z),
\end{eqnarray*}
and
\begin{eqnarray*}
  \left\vert \left[  1 - e(\tau_y^{(\varepsilon)}, \zeta) \right] e(\varepsilon^{-1}t^{(\varepsilon)}_y,\xi) \right\vert \leq \varepsilon
  \sum_{z\in \zeta} \phi(y-z)  e(\phi(y-\cdot), \xi).
\end{eqnarray*}
Then from (\ref{Vx}) for $\lambda$-almost all $\eta$ we have, see (\ref{o5}),
\begin{equation}
 \label{Vx2}
\left\vert\left(L_{\varepsilon, {\rm ren}}
-L_V \right)k(\eta)  \right\vert  \leq   \varepsilon \|k\|_{\vartheta_0} \left(\widetilde{C}  |\eta|e^{- \vartheta_0 |\eta|} +
  D (\eta) e^{- \vartheta_0 |\eta|}\right),
\end{equation}
with
\begin{eqnarray*}
\widetilde{C}  =  2 C \alpha  \|\phi\|_{L^\infty(\mathbb{R}^d)} \langle \phi \rangle  e^{-\vartheta_0}
\end{eqnarray*}
and
\begin{eqnarray*}
 D (\eta) & = &2 \alpha \exp\left( \langle \phi \rangle e^{-\vartheta_0} \right)
\|\phi\|_{L^\infty(\mathbb{R}^d)}  |\eta| (|\eta|+1). \\[.2cm]
 \end{eqnarray*}
Thus, we conclude that the expression in $( \cdot)$ in the right-hand
side of (\ref{Vx2}) is in $\mathcal{K}_\vartheta$, which yields (\ref{con}) and hence completes the proof.
\end{proof}

\appendix

\section{Proof of \eqref{0Q13}}

For  fixed $t\in (0,T(\vartheta))$ and $G_0\in B_{\rm bs}(\Gamma_0)$, by (\ref{0o12}) we have
\begin{equation}
  \label{y28a}
 \langle \! \langle G_0,k_t \rangle \! \rangle - \langle \! \langle G_0,k_t^{\Lambda_n, N_l} \rangle \! \rangle = \langle \! \langle G_t,k_0 \rangle \! \rangle - \langle \! \langle G_t,q_0^{\Lambda_n, N_l} \rangle \! \rangle  =: \mathcal{I}_n^{(1)}+
 \mathcal{I}_{n,l}^{(2)},
\end{equation}
where we set
\begin{eqnarray*}
 \mathcal{I}_n^{(1)}
 &= & \int_{\Gamma_0} G_t
(\eta)  k_0 (\eta) \bigl( 1 - \mathbb{I}_{\Gamma_{\Lambda_n}} (\eta)\bigr)\lambda (d \eta),
\\[.2cm]  \mathcal{I}_{n,l}^{(2)} & = & \int_{\Gamma_0}G_t(\eta) \Bigl[ k_0 (\eta)\mathbb{I}_{\Gamma_{\Lambda_n}} (\eta) -  q_0^{\Lambda_n, N_l} (\eta)\Bigr]\lambda (d \eta). \nonumber
\end{eqnarray*}
Let us prove that, for an arbitrary $\varepsilon >0$,
\begin{equation}
  \label{y29a}
 |\mathcal{I}_n^{(1)}|< \varepsilon/2,
\end{equation}
for sufficiently big $\Lambda_n$. Recall that $k_0$ is a correlation function, and hence is positive.
Taking into account that
\[
 \mathbb{I}_{\Gamma_\Lambda} (\eta) = \prod_{x\in \eta} \mathbb{I}_\Lambda (x),
\]
we write
\begin{eqnarray}
\label{Ap1}
|\mathcal{I}_n^{(1)}| & \leq & \int_{\Gamma_0} \left\vert G_t(\eta)\right\vert k_0 (\eta)
(1- \mathbb{I}_{\Gamma_{\Lambda_n}}(\eta)) \lambda (d \eta) \\[.2cm] & = & \sum_{p=1}^\infty \frac{1}{p!}
 \int_{(\mathbb{R}^d)^p}\left\vert(G_t)^{(p)} (x_1, \dots x_p)\right\vert k_0^{(p)}(x_1, \dots x_p)\nonumber \\[.2cm]
&& \times  J_{\Lambda_n} (x_1, \dots x_p) d x_1 \cdots d x_p,\nonumber
\end{eqnarray}
where
\begin{eqnarray}
 J_\Lambda (x_1 , \dots , x_p)  &:=&  1 - \mathbb{I}_\Lambda (x_1) \cdots \mathbb{I}_\Lambda (x_p) \nonumber\\[.2cm]
&:=&   \mathbb{I}_{\Lambda^c} (x_1) \mathbb{I}_\Lambda (x_2)\cdots \mathbb{I}_\Lambda (x_p) +
\mathbb{I}_{\Lambda^c} (x_2) \mathbb{I}_\Lambda (x_3)\cdots \mathbb{I}_\Lambda (x_p) \nonumber \\[.2cm]& & +
\cdots + \mathbb{I}_{\Lambda^c} (x_{p-1}) \mathbb{I}_{\Lambda} (x_p) + \mathbb{I}_{\Lambda^c} (x_p),
 \nonumber \\[.2cm] & \leq& \sum_{s=1}^p \mathbb{I}_{\Lambda^c} (x_{s}), \label{Ap2}
\end{eqnarray}
and $\Lambda^c := \mathbb{R}^d \setminus \Lambda$.
Taking into account that $k_0= k_{\mu_0}\in \mathcal{K}_{\vartheta_0}$, by (\ref{Ap2}) we obtain in (\ref{Ap1})
\begin{equation}
  \label{Ap3}
  |\mathcal{I}_n^{(1)}|  \leq \|k_0\|_{\alpha^*} \sum_{p=1}^\infty \frac{p}{p!} e^{-\alpha^* p} \int_{\Lambda_n^c} \int_{(\mathbb{R}^d)^{p-1}}
\Bigl\vert(G_t)^{(p)} (x_1, \dots x_p)\Bigr\vert  d x_1 \cdots d x_p.
\end{equation}
For $t$ as in (\ref{y28a}), one finds $\vartheta< \vartheta_0$ such that
$G_t \in \mathcal{G}_{\vartheta}$, cf. Theorem~\ref{1atm}. For this $\vartheta$ and $\varepsilon$ as in (\ref{y29a}), we pick $\bar{p}\in \mathbb{N}$
such that
\begin{equation}
\label{Ap4}
\sum_{p=\bar{p}+1}^\infty \frac{e^{-\vartheta p}}{p!}
 \int_{(\mathbb{R}^d)^p}\left\vert(G_t)^{(p)} (x_1, \dots x_p)\right\vert  d x_1 \cdots d x_p <   \frac{ \varepsilon e(\vartheta_0 - \vartheta)}{ 4 \|k_0\|_{\vartheta_0}}.
\end{equation}
Then we apply (\ref{Ap4}) and the following evident estimate
\[
p e^{-\vartheta_0 p} \leq e^{-\vartheta p} / e(\vartheta_0 - \vartheta),
\]
and obtain in (\ref{Ap3}) the following
\begin{eqnarray*}
|\mathcal{I}_n^{(1)}| & < & \frac{\|k_0\|_{\vartheta_0}}{e(\vartheta_0 - \vartheta)} \sum_{p=1}^{\bar{p}} \frac{p}{p!} e^{-\vartheta_0 p} \int_{\Lambda_n^c} \int_{(\mathbb{R}^d)^{p-1}}
\left\vert(G_t)^{(p)} (x_1, \dots x_p)\right\vert  d x_1 \cdots d x_p\\[.2cm] & + &  \varepsilon/4.
\end{eqnarray*}
Here the first term contains a finite number of summands, in each of which  $(G_t)^{(p)} $ is in $L^1 ((\mathbb{R}^d)^p)$.
Hence, it can be made strictly smaller than $\varepsilon/4$ by picking big enough
$\Lambda_n$, which yields (\ref{y29a}).

Let us show the same for the second integral in (\ref{y28a}).  Write, see (\ref{0K12}),  (\ref{Q2}), and (\ref{0Q1}),
\begin{eqnarray*}
\mathcal{I}^{(2)}_{n,l} & = & \int_{\Gamma_0} G_t(\eta) \int_{\Gamma_0} R^{\Lambda_n}_0 (\eta \cup \xi) \left[1- I_{N_l} (\eta \cup \xi)\right]\lambda (d\eta) \lambda (d \xi)\\ & = & \int_{\Gamma_0} F_t(\eta)R^{\Lambda_n}_0 (\eta ) \left[1- I_{N_l} (\eta)\right]\lambda (d\eta)  \nonumber \\[.2cm]
& = & \sum_{m=N_l+1}^\infty \frac{1}{m!} \int_{\Lambda_n^m}\left( R^{\Lambda_n}_0 \right)^{(m)} (x_1 , \dots , x_m) \nonumber \\[.2cm] & \times & F^{(m)}_t (x_1 , \dots , x_m)
dx_1 \cdots dx_m, \nonumber
\end{eqnarray*}
where
\[
F_t(\eta) := (KG_t)(\eta)= \sum_{\xi\subset \eta} G_t (\xi),
\]
and hence
\begin{equation}
  \label{y31}
F^{(m)}_t (x_1 , \dots , x_m)= \sum_{s=0}^m \ \sum_{\{i_1, \dots , i_s\} \subset \{1, \dots , m\}} \bigl( G_t\bigr)^{(s)}(x_{i_1}, \dots , x_{i_s}).
\end{equation}
By (\ref{0K12}), for $x_i\in \Lambda_n$, $i=1, \dots , m$, we have
\[
k_0^{(m)}(x_1 , \dots , x_m) = \sum_{s=0}^\infty \int_{\Lambda_n^s} \bigl( R^{\Lambda_n}_0 \bigr)^{(m+s)} (x_1 , \dots , x_m, y_1, \dots y_s) dy_1\cdots d y_s,
\]
from which we immediately get that
\begin{equation*}
\bigl( R^{\Lambda_n}_0 \bigr)^{(m)} (x_1 , \dots , x_m) \leq k_0^{(m)}(x_1 , \dots , x_m) \leq e^{-\vartheta_0m}\|k_0\|_{\vartheta_0},
\end{equation*}
since $k_0 \in \mathcal{K}_{\vartheta_0}$. Now let $\Lambda_n$ be such that (\ref{y29a}) holds.
Then we can have
\begin{equation}\label{y32a}
|\mathcal{I}^{(2)}_{n,l}| < \varepsilon /2,
\end{equation}
holding for big enough $N_l$ if
$e^{-\vartheta_0 |\cdot|} F_t$ is in $L^1 (\Lambda_n, d \lambda)$. By (\ref{y31}),
\begin{eqnarray*}
& &  \sum_{p=0}^\infty \frac{1}{p!} e^{-\vartheta_0 p} \int_{\Lambda_n^p} |F^{(p)} (x_1 , \dots , x_p)| dx_1 \cdots dx_p \\[.2cm]
& & \qquad \leq  \sum_{p=0}^\infty
 \sum_{s=0}^p \frac{1}{s! (p-s)!} e^{-\vartheta_0 s} \Bigl\Vert\bigl( G_t\bigr)^{(s)}\Bigr\Vert_{L^1(\Gamma_0, d\lambda)}e^{-\vartheta_0 (p-s)} [m(\Lambda_n)]^{p-s}\\[.2cm]
& & \qquad = \|G_t\|_{\vartheta_0} \exp\Bigl( e^{-\vartheta_0}m(\Lambda_n) \Bigr),
\end{eqnarray*}
where $m(\Lambda_n)$ is the Lebesgue measure of $\Lambda_n$, cf. (\ref{3A}). This yields (\ref{y32a}) and thereby also (\ref{0Q13}).

\paragraph{Acknowledgement}
The authors are cordially  grateful to the referee whose valuable and kind remarks were very helpful.
This work was financially supported by the DFG through SFB 701: ``Spektrale Strukturen
und Topologische Methoden in der Mathematik", the IGK ``Stochastics and Real World Models"  and through the
research project 436 POL 125/0-1, which is acknowledged by the authors.

\bibliographystyle{amsplain}

\end{document}